\documentclass[12pt]{article}

\usepackage{lmodern}
\usepackage[T1]{fontenc}
\usepackage{amsmath}
\usepackage{amsthm}
\usepackage{amssymb}
\usepackage{url}
\usepackage{latexsym}
\usepackage[small]{titlesec}
\usepackage[small,it]{caption}
\usepackage{mathtools} 

\usepackage[table]{xcolor}
\usepackage[margin=0.75in]{geometry}
\usepackage{mathrsfs}
\usepackage{graphicx}

\usepackage{tikz}
\usepackage{tkz-graph}
\usepackage{tkz-berge}
\usetikzlibrary{decorations.pathreplacing,decorations.markings}
\usetikzlibrary{cd}
\usepackage{diagbox}
\usepackage{pdflscape}

\usepackage[square,comma,numbers,sort&compress]{natbib}
\usepackage{lineno}

\usepackage[bookmarks]{hyperref}
\hypersetup{
	colorlinks=true,
	linkcolor=black,
	anchorcolor=black,
	citecolor=black,
	urlcolor=black,
	pdfpagemode=UseThumbs,
	pdftitle={},
	pdfsubject={Combinatorics},
	pdfauthor={Everett Sullivan},
}

\newcounter{todocounter}

\newcommand{\Z}{\mathbb{Z}}

\newcommand{\N}{\mathbb{N}}

\newcommand{\abs}[1]{\lvert {#1} \rvert}

\newcommand{\sumover}[3]{\sum_{{#1}}^{{#2}}{{#3}}}

\newcommand{\ignore}[1]{}
\newcommand*{\Scale}[2][4]{\scalebox{#1}{$#2$}}%

\theoremstyle{plain}
\newtheorem{theorem}{Theorem}[section]

\newtheorem{lemma}[theorem]{Lemma}
\newtheorem{corollary}[theorem]{Corollary}

\theoremstyle{definition}
\newtheorem{definition}[theorem]{Definition}
\newtheorem{example}[theorem]{Example}

\makeatletter
\newfont{\footsc}{cmcsc10 at 8truept}
\newfont{\footbf}{cmbx10 at 8truept}
\newfont{\footrm}{cmr10 at 10truept}
\renewenvironment{abstract}{
	\begin{list}{}%
	{\setlength{\rightmargin}{1in}%
	\setlength{\leftmargin}{1in}}%
	\item[]\ignorespaces\begin{small}}%
	{\end{small}\unskip\end{list}%
}

\newcommand*\patchAmsMathEnvironmentForLineno[1]{%
  \expandafter\let\csname old#1\expandafter\endcsname\csname #1\endcsname
  \expandafter\let\csname oldend#1\expandafter\endcsname\csname end#1\endcsname
  \renewenvironment{#1}%
     {\linenomath\csname old#1\endcsname}%
     {\csname oldend#1\endcsname\endlinenomath}}%
\newcommand*\patchBothAmsMathEnvironmentsForLineno[1]{%
  \patchAmsMathEnvironmentForLineno{#1}%
  \patchAmsMathEnvironmentForLineno{#1*}}%
\AtBeginDocument{%
\patchBothAmsMathEnvironmentsForLineno{equation}%
\patchBothAmsMathEnvironmentsForLineno{align}%
\patchBothAmsMathEnvironmentsForLineno{flalign}%
\patchBothAmsMathEnvironmentsForLineno{alignat}%
\patchBothAmsMathEnvironmentsForLineno{gather}%
\patchBothAmsMathEnvironmentsForLineno{multline}%
}

\newpagestyle{main}[\small]{
	\headrule
	\sethead[\usepage][][]
	{\sc Self Avoiding Random Walks}{}{\usepage}
}

\title{\sc Exactly-Solvable Self-Trapping Lattice Walks. Part I: Trapping in Ladder Graphs}

\author{%
	Alexander R. Klotz\\
	\small \texttt{alex.klotz@csulb.edu} \\
	\small Department of Physics and Astronomy\\[-3pt]
	\small California State University, Long Beach\\[-3pt]
	\small 1250 Bellflower Blvd., Long Beach, CA, 90840
	\and
	Everett Sullivan\\
	\small \texttt{everetts@vt.edu} \\
	\small Department of Mathematics\\[-3pt]
	\small Virginia Tech\\[-3pt]
	\small Blacksburg, Virginia USA \\
	\small Department of Mathematics\\[-3pt]
	\small Clayton State University\\[-3pt]
	\small Morrow, Georgia USA
}

\titleformat{\section}{\large\sc}{\thesection.}{1em}{}
\date{}

	\tikzset{
	  on each segment/.style={
	    decorate,
	    decoration={
	      show path construction,
	      moveto code={},
	      lineto code={
	        \path [#1]
	        (\tikzinputsegmentfirst) -- (\tikzinputsegmentlast);
	      },
	      curveto code={
	        \path [#1] (\tikzinputsegmentfirst)
	        .. controls
	        (\tikzinputsegmentsupporta) and (\tikzinputsegmentsupportb)
	        ..
	        (\tikzinputsegmentlast);
	      },
	      closepath code={
	        \path [#1]
	        (\tikzinputsegmentfirst) -- (\tikzinputsegmentlast);
	      },
	    },
	  },
	  mid arrow/.style={postaction={decorate,decoration={
	        markings,
	        mark=at position .7 with {\arrow[#1]{stealth}}
	      }}},
	}
 
\begin{document}
\maketitle

\pagestyle{main}

	\begin{abstract}

A growing self-avoiding walk (GSAW) is a stochastic process that starts from the origin on a lattice and grows by occupying an unoccupied adjacent lattice site at random. A sufficiently long GSAW will reach a state in which all adjacent sites are already occupied by the walk and become trapped, terminating the process. It is known empirically from simulations that on a square lattice, this occurs after a mean of 71 steps. In Part I of a two-part series of manuscripts, we consider simplified lattice geometries only two sites high (``ladders'') and derive generating functions for the probability distribution of GSAW trapping. We prove that a self-trapping walk on a square ladder will become trapped after a mean of 17 steps, while on a triangular ladder trapping will occur after a mean of 941/48$\approx$19.6 steps. We discuss additional implications of our results for understanding trapping in the ``infinite'' GSAW.
		
	\end{abstract}
	
	\section{Introduction and Background}\label{intro}

The self-avoiding walk (SAW) is a random walk that cannot visit the same location more than once. The statistics of SAWs are used to model the physics of polymer chains with excluded volume interactions \cite{mckenzie1971shape}. The universal quantities of SAWs, such as the growth exponent, size-scaling exponent, and finite-length corrections can be calculated by full enumerations of SAWs \cite{jensen2004enumeration,clisby2007self} or by stochastic generation of large ensembles \cite{clisby2010accurate} on various lattices. These universal SAW properties make predictions for the behavior of real polymers that can be observed in experiment, for example the dependence of radius of gyration on molecular weight \cite{cotton1980polymer}. Combinatorial techniques used to enumerate SAWs include the lace expansion method \cite{clisby2007self}, transfer matrix methods \cite{jensen2004enumeration}, and the length-doubling method \cite{schram2011exact}.

SAW statistics assume that each walk of a given length is equally likely to occur. An alternative ensemble is that of the growing self-avoiding walk (GSAW), a stochastic process in which a walk grows from the origin on a lattice by taking a random step to an adjacent unoccupied site. The probability of a given walk occurring is history-dependent and GSAWs have different statistics from SAWs \cite{lyklema1984growing}. One feature of GSAWs is trapping: a walk will eventually grow to a configuration at which there are no unoccupied adjacent sites, and the process terminates. 

It is known empirically that a growing self-avoiding walk on a square lattice will become trapped after a mean of approximately 71 steps, with a positively skewed distribution peaked at around 35 steps, and that walks are more likely to become trapped after an odd number of steps than even \cite{hemmer1984average}. It has been argued based on heuristic reasoning that walks become trapped with probability 1 \cite{hemmer1986trapping}. In addition, the trapping statistics of GSAWs have been explored on other lattices such as the triangular, honeycomb, and simple cubic lattices \cite{renner1996self}.

The mean trapping length is only known through stochastic simulations of GSAWs, there is no exact or closed-form expression for the probability distribution of trapping lengths. Deriving the mean trapping length requires knowledge of the probability of every possible self-avoiding walk. Enumerations are known up to, coincidentally, 71 steps, for which there are 4,190,893,020,903,935,054,619,120,005,916 distinct walks on the square lattice \cite{jensen2004enumeration}. It would not be feasible to calculate the probability of each walk in the ensemble. It is, in principle, possible to calculate the probability of every trapped configuration until the median and mode of the trapping probability distribution are surpassed, but the extensive computation required for that calculation would not provide sufficient insight to justify it.

The motivation behind this manuscript and its followup \cite{saws-part2} is to use a simplified system that allows exact solutions of the trapping behavior of growing self-avoiding walks through the use of combinatorial methods. We consider restricted lattices upon which growing self-avoiding walks may be trapped. Reduced latices prevent ``combinatorial explosion'' and allow a realistic number of cases to be considered, such that parameters such as the mean trapping length may be exactly solved. In particular, we consider lattices with square or triangular connectivity that only span two sites in one dimension, while remaining infinite or semi-infinite in the other, which we refer to as ladders. We derive the mean trapping length and other relevant metrics for these restricted lattices.

Beyond the mean trapping length, there are a number of empirical features of GSAW trapping statistics that we may exactly solve. It is known that on a square lattice, walks are more likely to become trapped after an odd number of steps, while the parity asymmetry is not observed in the triangular lattice, which has a mean of about 77 steps. The trapping probability distribution has a global maximum before decaying exponentially, with a different decay constant for square and triangular lattices \cite{renner1996self}. We recently explored trapping in GSAWs that were biased to take steps adjacent to occupied sites, modelling a poor-solvent interaction in polymer physics \cite{klotz}. A non-monotonic dependence of the trapping length on the bias strength was observed, with a global minimum for weak bias before it diverges exponentially. In this work, we seek to exactly solve these empirical phenomenon in a simplified system.

Confined walks adopt a highly extended configuration, and the statistics of confined random walks may be used to understand the physics of confined polymers, for example the mean extension of chains in a reptation tube \cite{odijk} or of DNA confined in nanochannel devices for genomic imaging \cite{dorfman}. Exactly-solvable models have proved insightful in this area. For example, a mapping between the physics of confined DNA in the extended de Gennes regime and the exactly solvable one-dimensional Domb-Joyce model has allowed an exact derivation of a molecule's mean extension \cite{mehlig}. In addition to deriving the trapping length and distribution of confined GSAWs, we also derive exact solutions for their mean extension to complement existing work in this area.

In this paper, Part I, we explore lattices that are two sites high and use pencil-and-paper combinatorial methods to derive a generating function for the trapping probability distribution. We prove that trapping is inevitable and derive the exact probability distribution and trapping length for the square and triangular lattice, the discussion of which begins on Page 29. In a forthcoming paper, Part II~\cite{saws-part2}, we use finite state machines and related methods to prove that the generating functions for the trapping probabilities on square lattices of arbitrary height are rational and we explicitly compute the trapping probability distribution for square lattices that are three and four sites high.

\section{Derivations}


		
	\subsection{Definitions and Notation}
	
	\begin{definition}
	    Given two graphs $G$ and $H$ we define the graph $G\Box H$, called the box product, to be the graph with vertex set $V = \{ (g,h) \mid g \in V(G) \text{ and } h \in V(H)\}$, and edge set
	    \begin{align*}
	        E = \{ ((g_{1},h_{1}),(g_{2},h_{2})) \mid & \left(h_{1} = h_{2} \text{ and } (g_{1},g_{2}) \in E(G)\right) \\
	                                                  & \text{ or } \left(g_{1} = g_{2} \text{ and } (h_{1},h_{2}) \in E(H)\right)\}
	    \end{align*}
	\end{definition}
	
	\begin{definition}
		We define the \textbf{path graph}, $P_{n}$ to be the simple graph with vertex set $V = \{1,\cdots, n\}$ and edge set $E = \{(a,b) \mid a,b \in V \text{ and } \abs{a-b} = 1\}$.
		
		Then the \textbf{ladder graph}, $L_{n}$ is $P_{2} \Box P_{n}$.
	\end{definition}
	
	\begin{example}
		Here are some examples of ladder graphs
		\begin{center}
			\begin{tikzpicture}[scale=1]
				\draw[help lines] (0,0) grid (0,1);
				\filldraw (0,0) circle (.1);
				\filldraw (0,1) circle (.1);
				\node at (0,-.5) {$L_{1}$};
			\end{tikzpicture}
			\qquad
			\begin{tikzpicture}[scale=1]
				\draw[help lines] (0,0) grid (1,1);
				\filldraw (0,0) circle (.1);
				\filldraw (1,0) circle (.1);
				\filldraw (0,1) circle (.1);
				\filldraw (1,1) circle (.1);
				\node at (0.5,-.5) {$L_{2}$};
			\end{tikzpicture}
			\qquad
			\begin{tikzpicture}[scale=1]
				\draw[help lines] (0,0) grid (3,1);
				\filldraw (0,0) circle (.1);
				\filldraw (1,0) circle (.1);
				\filldraw (2,0) circle (.1);
				\filldraw (3,0) circle (.1);
				\filldraw (0,1) circle (.1);
				\filldraw (1,1) circle (.1);
				\filldraw (2,1) circle (.1);
				\filldraw (3,1) circle (.1);
				\node at (1.5,-.5) {$L_{4}$};
			\end{tikzpicture}
		\end{center}
	\end{example}
	
	We will also consider the cases where the ladder graph is unbounded on one side.
	Thus we make the follow definitions
	\begin{definition}
		We define $P^{\ast}_{\infty}$ to be the simple graph with vertex set $V = \{a \mid a\in \Z \text{ and } a > 0\}$ with edge set $E = \{(a,b) \mid a,b \in V \text{ and } \abs{a-b} = 1\}$.
		
		We define $P_{\infty}$ to be the simple graph with vertex set $V = \Z$ with edge set $E = \{(a,b) \mid a,b \in V \text{ and } \abs{a-b} = 1\}$.
		
		Then we define $L^{\ast}_{\infty} = P_{2} \Box P^{\ast}_{\infty}$ and $L_{\infty} = P_{2} \Box P_{\infty}$
	\end{definition}
	
	\begin{example}
		The defined graphs are as follows:
		\begin{center}
			\begin{tikzpicture}[scale=1]
				\draw[help lines] (0,0) grid (5,1);
				\draw[help lines] (5,1) -- (5.25,1);
				\draw[help lines] (5,0) -- (5.25,0);
				\draw (5.5,0.5) node {$\scriptstyle\cdots\,$};
				\filldraw (0,0) circle (.1);
				\filldraw (1,0) circle (.1);
				\filldraw (2,0) circle (.1);
				\filldraw (3,0) circle (.1);
				\filldraw (4,0) circle (.1);
				\filldraw (5,0) circle (.1);
				\filldraw (0,1) circle (.1);
				\filldraw (1,1) circle (.1);
				\filldraw (2,1) circle (.1);
				\filldraw (3,1) circle (.1);
				\filldraw (4,1) circle (.1);
				\filldraw (5,1) circle (.1);
				\node at (2.5,-.5) {$L^{\ast}_{\infty}$};
			\end{tikzpicture}
		\end{center}
		\begin{center}
			\begin{tikzpicture}[scale=1]
				\draw[help lines] (0,0) grid (4,1);
				\draw[help lines] (4,1) -- (4.25,1);
				\draw[help lines] (4,0) -- (4.25,0) ;
				\draw[help lines] (0,1) -- (-.25,1);
				\draw[help lines] (0,0) -- (-.25,0) ;
				\draw (4.5,0.5) node {$\scriptstyle\cdots\,$};
				\draw (-.5,0.5) node {$\scriptstyle\cdots\,$};
				\filldraw (0,0) circle (.1);
				\filldraw (1,0) circle (.1);
				\filldraw (2,0) circle (.1);
				\filldraw (3,0) circle (.1);
				\filldraw (4,0) circle (.1);
				\filldraw (0,1) circle (.1);
				\filldraw (1,1) circle (.1);
				\filldraw (2,1) circle (.1);
				\filldraw (3,1) circle (.1);
				\filldraw (4,1) circle (.1);
				\node at (2,-.5) {$L_{\infty}$};
			\end{tikzpicture}
		\end{center}
	\end{example}
	
	We will consider Growing Self Avoiding Random Walks (GSAW) on these graphs.
	
	\begin{definition}
	    A \textbf{Growing Self Avoiding Random Walk} (GSAW) on a graph $G$ is a (possibly finite) sequence of vertices $W = \{v_{i}\}_{1}^{\infty}$ such that for all $i$, the vertices $v_{i}$ and $v_{i+1}$ are adjacent, and $v_{i} \neq v_{j}$ for any $i\neq j$.
	    
	    If $W$ is finite then additionally we say that the length of the path is the number of vertices in the path, and the last vertex must have all neighbors belonging to the path as well.
	\end{definition}
	
	We may think of this as starting at some vertex and randomly choosing another vertex to walk to that we have not been to before.
	If we arrive at a vertex where every neighbor is already in the path, then we must stop.
	
	Any step in a path of a lattice graph can be identified as a step in the positive $x$ direction which we will denote with the symbol $E$, a step in the negative $x$ direction which we will denote with the symbol $W$, a step in the negative $y$ direction which we will denote with the symbol $S$, or step in the positive $y$ direction which we will denote with the symbol $N$.
	
	Thus we may identify a path in a lattice graph by its starting point and a sequence of steps.
	On $L^{\ast}_{\infty}$ we will only consider GSAWs that start at $(1,1)$.
	Observe that by symmetry that this is same as starting at $(1,2)$.
	
	\begin{figure}
		\begin{center}
			\begin{tikzpicture}[scale=1]
				\draw[help lines] (0,0) grid (3,1);
				\draw[help lines] (3,1) -- (3.25,1);
				\draw[help lines] (3,0) -- (3.25,0);
				\draw (3.5,0.5) node {$\scriptstyle\cdots\,$};
				\filldraw (.125,.125) -- (-.125,.125) -- (-.125,-.125) -- (.125,-.125) -- cycle;
				\filldraw (1,0) ++(0,.125) -- ++(-.125,-.125) -- ++(.125,-.125) -- ++(.125,.125) -- ++(-.125,.125);
				\filldraw (2,0) circle (.1);
				\filldraw (3,0) circle (.1);
				\filldraw (0,1) circle (.1);
				\filldraw (1,1) circle (.1);
				\filldraw (2,1) circle (.1);
				\filldraw (3,1) circle (.1);
				\draw[ultra thick,postaction={on each segment={mid arrow=black}}] (0,0) -- (0,1) -- (1,1) -- (2,1) -- (2,0) -- (1,0);
				\node at (1.5,-.5) {$NEESW$};
			\end{tikzpicture}
		\end{center}
		\caption{An example of a GSAW on $L^{\ast}_{\infty}$.
			The starting position is denoted by a square and the ending position by a diamond.}
	\end{figure}
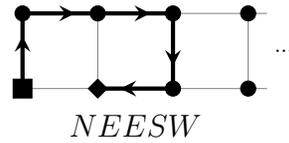
	
	\begin{definition}
		Let $P$ be a path on $L_{n},L^{\ast}_{\infty},$ or $L_{\infty}$.
		We define the \textbf{length} of a path to be the number of steps taken.
		We define the \textbf{width} of a path to be the (positive) distance between the largest and smallest $x$-coordinates of points on the path.
	\end{definition}
	
\subsection{Unbiased Square Ladders}

\subsubsection{Precursor Functions}
	
	\begin{definition}
		We define a hook path of order $n$ on $L^{\ast}_{\infty}$ to be the path
		\begin{displaymath}
			H^{\ast}_{n} = \underbrace{E\cdots E}_{\text{$n$ steps}}N\underbrace{W\cdots W}_{\text{$n$ steps}}.
		\end{displaymath}
	\end{definition}
	
	We observe that for all $n > 0$, $H^{\ast}_{n}$ is a GSAW.
	Given a GSAW that selects neighbors uniformly at random, the probability of a GSAW using a  being $H^{\ast}_{n}$ is
	\begin{displaymath}
		\left(\frac{1}{2}\right)^{n}\cdot\frac{1}{2}\cdot\frac{1}{2} = \left(\frac{1}{2}\right)^{n+2}.
	\end{displaymath}
	The length of $H^{\ast}_{n}$ is $2n+1$ and its width is $n$
	
	\begin{lemma}
		Let $h_{\ell,w}$ be the probability that a GSAW on $L^{\ast}_{\infty}$ terminates as a hook path of length $\ell$ and width $w$.
		Then the bivariate ordinary generating function $H^{\ast}(x,y)$ of $h_{\ell,w}$ is
		\begin{displaymath}
			H^{\ast}(x,y) = \frac{x^{3}y}{8-4x^{2}y}.
		\end{displaymath}
	\end{lemma}
	
	\begin{proof}
		Since we are only considering hook paths that terminate, $H_{0}$ is not considered.
		Thus for the remaining hook paths $H^{\ast}_{n}$, each has length $2n+1$ and probability $(1/2)^{n+2}$.
		We observe that for any hook path, $\ell = 2w+1$.
		Thus we see that the length of every hook path must be odd and have minimum length 3.
		Thus if $\ell$ is even, if $\ell$ is 1 $h_{\ell,w} = 0$, $h_{1,w} = 0$, and otherwise $h_{\ell,w} = (1/2)^{(\ell+3)/2}$ otherwise.
		Thus
		\begin{align*}
			H^{\ast}(x,y) & = \sumover{\ell=0}{\infty}{\sumover{w=0}{\infty}{h_{\ell,w}x^{\ell}y^{w}}} 
			        = \sumover{w=0}{\infty}{h_{2w+1,w}x^{2w+1}y^{w}} 
			        = \sumover{k=1}{\infty}{h_{2k+1,k}x^{2k+1}y^{k}} \\
			      &  = \sumover{k=1}{\infty}{\left(\frac{1}{2}\right)^{k+2}x^{2k+1}y^{k}} 
			        = \sumover{k=0}{\infty}{\left(\frac{1}{2}\right)^{k+3}x^{2k+3}y^{k+1}} \\
			      &  = \left(\frac{x}{2}\right)^{3}y\sumover{k=0}{\infty}{\left(\frac{x^{2}y}{2}\right)^{k}} 
			        = \frac{\left(\frac{x}{2}\right)^{3}y}{1 - \frac{x^{2}y}{2}} 
			        =\boxed{\frac{x^{3}y}{8-4x^{2}y}}.
		\end{align*}
		
	\end{proof}
	
	\begin{definition}
		We define a twist path of order $n$ on $L^{\ast}_{\infty}$ be the path
		\begin{displaymath}
			T^{\ast}_{n} = \underbrace{E\cdots E}_{\text{$n$ steps}}NE.
		\end{displaymath}
	\end{definition}
	
	The probability of the first $n$ steps of a GSAW being $T^{\ast}_{n}$ is
	\begin{displaymath}
		\left(\frac{1}{2}\right)^{n+2} \quad \text{if $n \geq 1$, and} \quad \frac{1}{2} \quad \text{if $n = 0$.}
	\end{displaymath}
	The length of $T^{\ast}_{n}$ is $n+2$ and its width is $n+1$.
	
	\begin{lemma}
		Let $t_{\ell,w}$ be the probability that the first $\ell$ steps of a GSAW is a twist path of length $\ell$ and width $w$.
		Then the bivariate ordinary generating function $T^{\ast}(x,y)$ of $t_{\ell,w}$ is
		\begin{displaymath}
			T^{\ast}(x,y) = \frac{x^{2}y}{2} + \frac{x^{3}y^{2}}{8-4xy}.
		\end{displaymath}
	\end{lemma}
	
	\begin{proof}
		We observe that the minimum length of a twist path is 2, and that for any twist path that $\ell = 1 + w$.
		Thus
				\begin{align*}
			T^{\ast}(x,y) & = \sumover{\ell=0}{\infty}{\sumover{w=0}{\infty}{t_{\ell,w}x^{\ell}y^{w}}}
			        = \sumover{\ell=2}{\infty}{t_{\ell,\ell-1}x^{\ell}y^{\ell-1}} 
			       = t_{2,1}x^{2}y + \sumover{\ell=3}{\infty}{t_{\ell,\ell-1}x^{\ell}y^{\ell-1}} \\
			       & = t_{2,1}x^{2}y + \sumover{k=0}{\infty}{t_{k+3,k+2}x^{k+3}y^{k+2}} 
			        = \frac{x^{2}y}{2} + \sumover{k=0}{\infty}{\left(\frac{1}{2}\right)^{k+3}x^{k+3}y^{k+2}} 
			         \\
			       &= \frac{x^{2}y}{2} + \frac{x^{3}y^{2}}{8}\sumover{k=0}{\infty}{\left(\frac{xy}{2}\right)^{k}} = \frac{x^{2}y}{2} + \frac{\left(\frac{x^{3}y^{2}}{8}\right)}{1 - \frac{xy}{2}} 
			        =\boxed{ \frac{x^{2}y}{2} + \frac{x^{3}y^{2}}{8-4xy}}.
		\end{align*}
	\end{proof}
		

	\begin{figure}
		\begin{center}
			\begin{tikzpicture}[scale=1]
				\draw[help lines] (0,0) grid (4,1);
				\draw[help lines] (4,1) -- (4.25,1);
				\draw[help lines] (4,0) -- (4.25,0);
				\draw (4.5,0.5) node {$\scriptstyle\cdots\,$};
				\fill[black!50!white] (-.25 ,-.25) rectangle (2.5,1.25);
				\draw[very thick] (-.25 ,-.25) rectangle (2.5,1.25);
				\filldraw (.125,.125) -- (-.125,.125) -- (-.125,-.125) -- (.125,-.125) -- cycle;
				\filldraw (3,1) ++(0,.125) -- ++(-.125,-.125) -- ++(.125,-.125) -- ++(.125,.125) -- ++(-.125,.125);
				\filldraw (1,0) circle (.1);
				\filldraw (2,0) circle (.1);
				\filldraw (3,0) circle (.1);
				\filldraw (4,0) circle (.1);
				\filldraw (0,1) circle (.1);
				\filldraw (1,1) circle (.1);
				\filldraw (2,1) circle (.1);
				\filldraw (4,1) circle (.1);
				\draw[ultra thick,postaction={on each segment={mid arrow=black}}] (0,0) -- (1,0) -- (2,0) -- (2,1) -- (3,1);
			\end{tikzpicture}
		\end{center}
		\caption{If the first $n$ steps form $T^{\ast}_{n}$, then the remainder of the path must be another GSAW since the path can't move back, and so by symmetry it is in the same starting position.}
		\label{figure:one-sided-ladder}
	\end{figure}
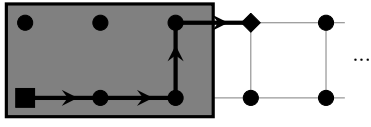
	
\subsubsection{The One Sided Case}
	
	\begin{theorem}
		Let $p_{\ell,w}$ be the probability that a GSAW on $L^{\ast}_{\infty}$ terminates having taken $\ell$ steps and a width of $w$.	
		Then the bivariate ordinary generating function $L^{\ast}_{\infty}(x,y)$ of $p_{\ell,w}$ is
		\begin{displaymath}
			L^{\ast}_{\infty}(x,y) = \frac{x^{3} y (x y - 2)}{(x^{2}y - 2)(x^{3}y^{2} -4x^{2}y - 4 xy + 8)}.
		\end{displaymath}
	\end{theorem}
	
	\begin{proof}
		We first consider an arbitrary GSAW on $L^{\ast}_{\infty}$.
		We observe that with probability 1, the GSAW makes a $N$ step at some point, since the only way to avoid such a step is an infinite series of $E$ steps which has probability 0 of happening.
		After the GSAW makes the $N$ step, it must either take an $E$ step or $W$ step.
		If it takes a $W$ step then it forms a hook path and terminates.
		If it takes an $E$ step then it forms a twist path and must continue.
		However, at the end of the twist path, the remaining spaces reachable from the current position is equivalent to starting a new GSAW, as seen in figure \ref{figure:one-sided-ladder}.
		
		Thus every GSAW on $L^{\ast}_{\infty}$ is either a hook path or a twist path followed by another GSAW on $L^{\ast}_{\infty}$.
		
		Let $p_{\ell,w}$ be the probability that a GSAW on $L^{\ast}_{\infty}$ terminates with length $\ell$ and width $w$.
		Then
		\begin{displaymath}
			p_{\ell,w} = \sumover{i=0}{\ell}{\sumover{j=0}{\infty}t_{i,j}p_{\ell - i,w - j}} + h_{\ell,w}.
		\end{displaymath}
		Thus
		\begin{displaymath}
			L^{\ast}_{\infty}(x,y) = T^{\ast}(x,y)L^{\ast}_{\infty}(x,y) + H^{\ast}(x,y).
		\end{displaymath}
		Solving for $L^{\ast}_{\infty}(x,y)$ we get that
		\begin{displaymath}
			L^{\ast}_{\infty}(x,y) = \frac{H^{\ast}(x,y)}{1-T^{\ast}(x,y)} = \frac{x^{3} y (x y - 2)}{(x^{2}y - 2)(x^{3}y^{2} -4x^{2}y - 4 xy + 8)}.
		\end{displaymath}
	\end{proof}
	
	\begin{corollary}
	    The generating function for the lengths of paths of terminating GSAWs on $L^{\ast}_{\infty}$ is
	    \begin{displaymath}
	        L^{\ast}_{\infty}(x,1) = \frac{x^{3} (x - 2)}{(x^{2} - 2)(x^{3} -4x^{2} - 4 x + 8)}.
	    \end{displaymath}
	    The generating function for the widths of paths of terminating GSAWs on $L^{\ast}_{\infty}$ is
	    \begin{displaymath}
	        L^{\ast}_{\infty}(1,y) = \frac{y (y - 2)}{(y - 2)(y^{2} -8y + 8)}
	    \end{displaymath}
	\end{corollary}
	
	\begin{corollary}
		The probability that a GSAW on $L^{\ast}_{\infty}$ terminates is 1.
	\end{corollary}
	
	\begin{proof}
		Let $p$ be the probability that a GSAW on $L^{\ast}_{\infty}$ terminates and $q$ be the probability that it doesn't.
		If the GSAW terminates, then it has a finite number of steps.
		Thus
		\begin{displaymath}
			p = \sumover{\ell=0}{\infty}{\sumover{w=0}{\infty}{p_{\ell,w}}} = L^{\ast}_{\infty}(1,1) = 1.
		\end{displaymath}
		Thus $q=0$ and a GSAW will terminate with probability 1.
	\end{proof}
	
	\begin{corollary}
		The expected number of steps taken before a GSAW on $L^{\ast}_{\infty}(x,y)$ terminates is 13 with variance 98 (standard deviation $\approx 3.6$).
		The expected width when a GSAW on $L^{\ast}_{\infty}(x,y)$ terminates is 7 with variance 40 (standard deviation $\approx 6.3$).
	\end{corollary}
	
	\begin{proof}
		Let $p_{\ell}$ be the probability that a GSAW on $L^{\ast}_{\infty}$ terminates with $\ell$ steps.
		Then
		\begin{displaymath}
			p_{\ell} = \sumover{w=0}{\infty}{p_{\ell,w}}.
		\end{displaymath}
		The expected value of the number of steps is
		\begin{displaymath}
			\mathbb{E}(\ell) = \sumover{\ell = 0}{\infty}{\ell p_{\ell}}.
		\end{displaymath}
		Thus 
		\begin{displaymath}
			\mathbb{E}(\ell) = \left.\left(\frac{\partial}{\partial x}L^{\ast}_{\infty}(x,1)\right)\right|_{x=1} = 13
		\end{displaymath}
		and the variance of the number of steps is
		\begin{align*}
			\mathbb{V}(\ell) & = \left.\left(\frac{\partial^{2}}{\partial x^{2}}L^{\ast}_{\infty}(x,1)\right)\right|_{x=1} + \left.\left(\frac{\partial}{\partial x}L^{\ast}_{\infty}(x,1)\right)\right|_{x=1} - \left(\left.\left(\frac{\partial}{\partial x}L^{\ast}_{\infty}(x,1)\right)\right|_{x=1}\right)^{2} \\
			& = 98.
		\end{align*}
		
		By similar analysis the expected value of the displacement is
		\begin{displaymath}
			\mathbb{E}(w) = \left.\left(\frac{\partial}{\partial y}L^{\ast}_{\infty}(1,y)\right)\right|_{y=1} = 7
		\end{displaymath}
		and the variance of the displacement is
		\begin{align*}
			\mathbb{V}(w) & = \left.\left(\frac{\partial^{2}}{\partial y^{2}}L^{\ast}_{\infty}(1,y)\right)\right|_{y=1} + \left.\left(\frac{\partial}{\partial y}L^{\ast}_{\infty}(1,y)\right)\right|_{y=1} - \left(\left.\left(\frac{\partial}{\partial y}L^{\ast}_{\infty}(1,y)\right)\right|_{y=1}\right)^{2} \\
			& = 40.
		\end{align*}
	\end{proof}
	
\subsubsection{The Double Sided Case}
	
	\begin{definition}
		We define a hook path of order $n$ on $L_{\infty}$ to be the path
		\begin{displaymath}
			H_{n} = \underbrace{E\cdots E}_{\text{$n$ steps}}N\underbrace{W\cdots W}_{\text{$n+1$ steps}}.
		\end{displaymath}
	\end{definition}

	The probability of GSAW having $H_{n}$ as its first segment is
	\begin{displaymath}
		\frac{1}{3\cdot 2^{n}}.
	\end{displaymath}
	The length of $H_{n}$ is $2n + 2$ and its width is $n+1$.
	
	\begin{lemma}
		Let $h_{\ell,w}$ be the probability that a GSAW on $L_{\infty}$ terminates as a hook path of length $\ell$ and width $w$.
		Then the bivariate ordinary generating function $H(x,y)$ of $h_{\ell,w}$ is
		\begin{displaymath}
			H(x,y) = \frac{2x^{2}y}{6-3x^{2}y}.
		\end{displaymath}
	\end{lemma}
	
	\begin{proof}
	
		We first observe by symmetry that every starting position on $L_{\infty}$ is the same.
		Without loss of generality we assume we start at $(0,1)$.
		The first step is either a $N,W$, or $E$ step.
		However by symmetry $W$ and $E$ produce the same path, so without loss of generality we may either move $N$ or $E$.
		
		Thus a path has $1/3$ probability of taking a $N$ step at which point the next step must form $H_{n}$ by symmetry.
		Therefor the probability of a GSAW having $H_{0}$ as its first segment is $1/3$.
		
		Otherwise the path takes an $E$ step with probability $2/3$, at which point to create $H_{n}$ on $L_{\infty}$, the path must move $n-1$ $E$ steps, each with probability $1/2$.
		Then a $N$ step with probability $1/2$ and then a $W$ step with probability $1/2$, at which point the remaining $W$ steps must be taken since there are no other options.
		Thus the probability of a GSAW having $H_{n}$ as its first segment is
		\begin{displaymath}
		    \frac{2}{3}\left(\frac{1}{2}\right)^{n-2}\cdot\frac{1}{2}\cdot\frac{1}{2} = \frac{1}{3\cdot 2^{n}}.
		\end{displaymath}

		We observe that for any hook path, $\ell = 2w$.
		Thus we see that the length of every hook path must be even and have minimum length 2.
		Thus if $\ell$ is odd, $h_{\ell,w} = 0$.
		Thus
		\begin{align*}
			H(x,y) & = \sumover{\ell=0}{\infty}{\sumover{w=0}{\infty}{h_{\ell,w}x^{\ell}y^{w}}} 
			       = \sumover{w=0}{\infty}{h_{2w,w}x^{2w}y^{w}} 
			        = \sumover{k=1}{\infty}{h_{2k,w}x^{2k}y^{k}} \\
			      &  = \sumover{k=1}{\infty}{\frac{1}{3\cdot 2^{k-1}}x^{2k}y^{k}} 
			        = \sumover{k=0}{\infty}{\frac{1}{3\cdot 2^{k}}x^{2k+2}y^{k+1}} \\
			       & = \frac{x^2y}{3}\sumover{k=0}{\infty}{\left(\frac{x^{2}y}{2}\right)^{k}} 
			        = \frac{\frac{x^2y}{3}}{1-\frac{x^{2}y}{2}} = 
			        \boxed{ \frac{2x^{2}y}{6-3x^{2}y}}.
		\end{align*}
	\end{proof}

	
	\begin{definition}
		We define a twist path of order $n$ on $L_{\infty}$ to be the path
		\begin{displaymath}
			T_{n} = \underbrace{E\cdots E}_{\text{$n$ steps}}ENE.
		\end{displaymath}
	\end{definition}

	The probability of GSAW having $T_{n}$ as its first segment is
	\begin{displaymath}
		\frac{1}{3\cdot 2^{n+1}}.
	\end{displaymath}
	The length of $T_{n}$ is $3+n$ and its width is $2+n$.
	
	\begin{lemma}
		Let $t_{\ell,w}$ be the probability that a GSAW begins with a twist path of length $\ell$ and width $w$.
		Then the bivariate ordinary generating function $T(x,y)$ of $t_{\ell,w}$ is
		\begin{displaymath}
			T(x,y) = \frac{x^{3}y^{2}}{6-3xy}.
		\end{displaymath}
	\end{lemma}
	
	\begin{proof}
		We first observe by symmetry that every starting position on $L_{\infty}$ is the same.
		Without loss of generality we assume we start at $(0,1)$.
		By symmetry the first step must be a $W$ or $E$ step.
		The probability of the first step being one of these two is $2/3$.
		By symmetry we may assume without loss of generality that the step was an $E$ step.
		
		After the first step, in order to create $T_{n}$ on $L_{\infty}$, the path must move $n$ $E$ steps each with probability $1/2$.
		Then a $N$ step with probability $1/2$ and then an $E$ step with probability $1/2$.
		At which point the remaining $W$ steps must be taken since there are no other options.
		Thus the probability of a GSAW having $T_{n}$ as its first segment is $(2/3)(1/2)^{n}(1/2)(1/2) = 1/(3\cdot 2^{n+1})$.
		
		We observe that for any hook path, $\ell = w+1$ and the minimum width is $2$.
		Thus
		\begin{align*}
			T(x,y) & = \sumover{\ell=0}{\infty}{\sumover{w=0}{\infty}{t_{\ell,w}x^{\ell}y^{w}}} 
			        = \sumover{w=0}{\infty}{t_{w+1,w}x^{w+1}y^{w}} 
			        = \sumover{w=0}{\infty}{\frac{1}{3\cdot 2^{w-1}}x^{w+1}y^{w}} \\
			       & = \sumover{w=2}{\infty}{\frac{1}{3\cdot 2^{w-1}}x^{w+1}y^{w}} 
			        = \sumover{k=0}{\infty}{\frac{1}{3\cdot 2^{k+1}}x^{k+3}y^{k+2}} 
			        = \frac{x^{3}y^{2}}{6}\sumover{k=0}{\infty}{\left(\frac{xy}{2}\right)^{k}} \\
			       & = \frac{\frac{x^{3}y^{2}}{6}}{1-\frac{xy}{2}}
			        =\boxed{ \frac{x^{3}y^{2}}{6-3xy}}.
		\end{align*}
		
		
	\end{proof}
	
	\begin{theorem}
		let $p_{\ell,w}$ be the probability that a GSAW on $L_{\infty}$ terminates having taken $\ell$ steps and a displacement of $w$.	
		Then the bivariate ordinary generating function $L_{\infty}(x,y)$ of $p_{\ell,w}$ is
		\begin{displaymath}
			L_{\infty}(x,y) = \frac{x^5 y^{2} (4 - x^3 y)}{3 (x^{2}y - 2)^2 (x^3 y^2 - 4 x^2 y - 4 x y + 8)}.
		\end{displaymath}
	\end{theorem}
	
	\begin{proof}
		We first consider an arbitrary GSAW on $L_{\infty}$.
		
		With probability 1, the GSAW will take a north step, and thus form either a hook or twist path on $L_{\infty}$.
		However, after that, the path continues as if it were on $L^{\ast}_{\infty}$ since it can no longer go back.
		
		Thus every GSAW on $L_{\infty}$ is either a hook path or a twist path followed by another GSAW on $L^{\ast}_{\infty}$.
		
		Thus
		\begin{displaymath}
			L_{\infty}(x,y) = (T(x,y)+ H(x,y))L^{\ast}_{\infty}(x,y).
		\end{displaymath}
		Since the other generating functions have already been computed we get that
		\begin{displaymath}
			L_{\infty}(x,y) = \frac{x^5 y^{2} (4 - x^3 y^{2})}{3 (x^{2}y - 2)^2 (x^3 y^2 - 4 x^2 y - 4 x y + 8)}.
		\end{displaymath}
	\end{proof}
	\begin{corollary}\label{theorem:trapping-length-square-lattice}
	    The generating function for the lengths of paths of terminating GSAWs (the trapping length) on $L_{\infty}$ is
	    \begin{displaymath}
	        L_{\infty}(x,1) = \frac{x^5 (4 - x^3)}{3 (x^{2} - 2)^2 (x^3 - 4 x^2 - 4 x + 8)}.
	    \end{displaymath}
	    The generating function for the widths of paths of terminating GSAWs on $L_{\infty}$ is
	    \begin{displaymath}
	        L_{\infty}(1,y) = \frac{y^{2} (4 - y)}{3 (y - 2)^2 (y^2 - 8y + 8)}
	    \end{displaymath}
	\end{corollary}
	
	\begin{corollary}
		The probability that a GSAW on $L_{\infty}$ terminates is 1.
	\end{corollary}
	
	\begin{proof}
		Let $p$ be the probability that a GSAW on $L_{\infty}$ terminates and $q$ be the probability that it doesn't.
		If the GSAW terminates, then it has a finite number of steps.
		Thus
		\begin{displaymath}
			p = \sumover{\ell=0}{\infty}{\sumover{d=0}{\infty}{p_{\ell,d}}} = L_{\infty}(1,1) = 1.
		\end{displaymath}
		Thus $q=0$ and a GSAW will terminate with probability 1.
	\end{proof}
	
	\begin{corollary}
		The expected number of steps taken before a GSAW on $L_{\infty}$ terminates is 17 with variance 104 (standard deviation $\approx 10.2$).
		The expected width when a GSAW on $L_{\infty}$ terminates is $28/3 \approx 9.6$ with variance $740/9 \approx 82.2$ (standard deviation $\approx 9.1$).
	\end{corollary}
	
\subsection{Unbiased Triangular Ladders}

	\subsubsection{Definitions}
	
	\begin{definition}
		We define the graph $\Delta_{\infty}$ to be the graph with vertex set $V = \{(a,b) \mid a \in \Z, \quad b\in \{0,1\}, \text{ and } a+b \text{ is even}\}$ and edge set $E = \{((x_{1},y_{1}),(x_{2},y_{2})) \mid (\abs{x_{1}-x_{2}} = 1 \text{ and } \abs{y_{1}-y_{2}} = 1) \text{ or } (y_{1} = y_{2} \text{ and } \abs{x_{1}-x_{2}} = 2)\}$
	\end{definition}
		
	\begin{definition}
		We define the graph $\Delta^{\ast}_{\infty}$ to be the induced sub-graph of $\Delta_{\infty}$ whose vertex set is given by $V = \{(a,b) \mid a \in \N, \quad b\in \{0,1\}, \text{ and } a+b \text{ is even}\}$
	\end{definition}
	
	\begin{example}
		The defined graphs are as follows:
		\begin{center}
			\begin{tikzpicture}[scale=1]
				\draw[help lines] (0,0) -- (4.25,0);
				\draw[help lines] (0.5,1) -- (4.75,1);
				\draw[help lines] (0,0) -- (0.5,1) -- (1,0) -- (1.5,1) -- (2,0) -- (2.5,1) -- (3,0) -- (3.5,1) -- (4,0) -- (4.5,1) ;
				\draw (4.75,0.5) node {$\scriptstyle\cdots\,$};
				\filldraw (0,0) circle (.1);
				\filldraw (1,0) circle (.1);
				\filldraw (2,0) circle (.1);
				\filldraw (3,0) circle (.1);
				\filldraw (4,0) circle (.1);
				\filldraw (0.5,1) circle (.1);
				\filldraw (1.5,1) circle (.1);
				\filldraw (2.5,1) circle (.1);
				\filldraw (3.5,1) circle (.1);
				\filldraw (4.5,1) circle (.1);
				\node at (2.5,-.5) {$\Delta^{\ast}_{\infty}$};
			\end{tikzpicture}
		\end{center}
		\begin{center}
			\begin{tikzpicture}[scale=1]
				\draw[help lines] (-.25,0) -- (4.25,0);
				\draw[help lines] (0.25,1) -- (4.75,1);
				\draw[help lines] (0,0) -- (0.5,1) -- (1,0) -- (1.5,1) -- (2,0) -- (2.5,1) -- (3,0) -- (3.5,1) -- (4,0) -- (4.5,1) ;
				\draw (4.75,0.5) node {$\scriptstyle\cdots\,$};
				\draw (-.25,0.5) node {$\scriptstyle\cdots\,$};
				\filldraw (0,0) circle (.1);
				\filldraw (1,0) circle (.1);
				\filldraw (2,0) circle (.1);
				\filldraw (3,0) circle (.1);
				\filldraw (4,0) circle (.1);
				\filldraw (0.5,1) circle (.1);
				\filldraw (1.5,1) circle (.1);
				\filldraw (2.5,1) circle (.1);
				\filldraw (3.5,1) circle (.1);
				\filldraw (4.5,1) circle (.1);
				\node at (2.5,-.5) {$\Delta_{\infty}$};
			\end{tikzpicture}
		\end{center}
	\end{example}
	
	\begin{figure}
		\begin{center}
			\begin{tikzpicture}[scale=1]
				\draw[help lines] (0,0) -- (4.25,0);
				\draw[help lines] (0.5,1) -- (4.75,1);
				\draw[help lines] (0,0) -- (0.5,1) -- (1,0) -- (1.5,1) -- (2,0) -- (2.5,1) -- (3,0) -- (3.5,1) -- (4,0) -- (4.5,1) ;
				\draw (4.75,0.5) node {$\scriptstyle\cdots\,$};
				\filldraw (0,0) circle (.1);
				\filldraw (1,0) circle (.1);
				\filldraw (2,0) circle (.1);
				\filldraw (3,0) circle (.1);
				\filldraw (4,0) circle (.1);
				\filldraw (0.5,1) circle (.1);
				\filldraw (1.5,1) circle (.1);
				\filldraw (2.5,1) circle (.1);
				\filldraw (3.5,1) circle (.1);
				\filldraw (4.5,1) circle (.1);
				\draw (0.5,1) circle (.2);
				\draw (0,0) circle (.2);
				\node [label={[label distance=.3]150:\small{wide corner}}] at (0.5,1) {};
				\node [label={[label distance=.3]210:\small{narrow corner}}] at (0,0) {};
				\node at (2.5,-.5) {$EN_{d}ES_{u}N_{u}ES_{d}$};
				\draw[ultra thick,postaction={on each segment={mid arrow=black}}] (0,0) -- (1,0) -- (0.5,1) -- (1.5,1) -- (2,0) -- (2.5,1) -- (3.5,1) -- (3,0);
			\end{tikzpicture}
		\end{center}
		\caption{The $N_{u}$ and $S_{u}$ steps will always travel on the diagonal that moves east, while the $N_{d}$ and $S_{d}$ steps will always travel on the diagonal that moves west.}
		\label{figure:one-sided-triangle}
	\end{figure}
	
	Since GSAWs on these graphs are no longer on a rectangular grid, we have different notation for the representation of steps.
	
	Any step in a path of the graphs $\Delta^{\ast}_{\infty}$ and $\Delta_{\infty}$ can be identified as a step in the positive $x$ direction which we will denote with the symbol $E$, a step in the negative $x$ direction which we will denote with the symbol $W$, a step in the negative $x$ and $y$ direction which we will denote with the symbol $S_{d}$, or step in the positive $y$ direction and negative $x$ which we will denote with the symbol $N_{d}$, a step in the positive $x$ and $y$ direction which we will denote with the symbol $N_{u}$, or step in the negative $y$ direction and positive $x$ which we will denote with the symbol $S_{u}$.
	
	All of these steps can be seen in the path given by Figure~\ref{figure:one-sided-triangle}.
	
	We additionally define the vertex $(0,0)$ in $\Delta^{\ast}_{\infty}$ to be the ``narrow corner'' and the vertex $(1,1)$ in $\Delta^{\ast}_{\infty}$ to be the ``wide corner''. 
	
	\subsubsection{Precursor Functions}
	
	\begin{definition}
		We define an upward twist path of order $n$ on $\Delta^{\ast}_{\infty}$ starting at the wide corner to be the path
		\begin{displaymath}
			T^{u}_{n} = \underbrace{E\cdots E}_{\text{$n$ steps}}S_{u}E.
		\end{displaymath}
	\end{definition}

	The probability of GSAW on $\Delta^{\ast}_{\infty}$ starting at the wide corner having $T^{u}_{n}$ as its first segment is
	\begin{displaymath}
		\frac{1}{3^{n+2}}.
	\end{displaymath}
	The length of $T^{u}_{n}$ is $n+2$.
	
	\begin{lemma}
		Let $t^{u}_{\ell}$ be the probability that a GSAW on $\Delta^{\ast}_{\infty}$ starting at the wide corner has $T^{u}_{n}$ as its first segment where $n+2 = \ell$.
		Then the ordinary generating function $T^{u}(x)$ of $t^{u}_{\ell}$ is
		\begin{displaymath}
			T^{u}(x) = \frac{x^{2}}{9-3x}.
		\end{displaymath}
	\end{lemma}
	
	\begin{proof}
		
		In order to create $T^{u}_{n}$ on $\Delta^{\ast}_{\infty}$, the path must move $n$ $E$ steps each with probability $1/3$, then a $S_{u}$ step with probability $1/3$ and then an $E$ step with probability $1/3$.
		Thus the probability of a GSAW on $\Delta^{\ast}_{\infty}$ starting at the wide corner having $T^{u}_{n}$ as its first segment is $1/(3^{n+2})$.
		
		Thus
		\begin{align*}
			T^{u}(x) & = \sumover{\ell=0}{\infty}{t^{u}_{\ell}x^{\ell}} 
			          = \sumover{\ell=2}{\infty}{t^{u}_{\ell}x^{\ell}} 
			          = \sumover{n=0}{\infty}{t^{u}_{n+2}x^{n+2}} \\
			         & = \sumover{n=0}{\infty}{\frac{1}{3^{n+2}}x^{n+2}} 
			          = \frac{x^{2}}{9}\sumover{n=0}{\infty}{\left(\frac{x}{3}\right)^{n}} 
			         = \frac{\frac{x^{2}}{9}}{1 - \frac{x}{3}} 
			         = \boxed{\frac{x^{2}}{9-3x}}.
		\end{align*}\
	\end{proof}
	
	\begin{definition}
		We define a crooked path of order $n$ on $\Delta^{\ast}_{\infty}$ starting at the wide corner to be the path
		\begin{displaymath}
			C_{n} = \underbrace{E\cdots E}_{\text{$n$ steps}}S_{u}N_{u}.
		\end{displaymath}
	\end{definition}
	
	\begin{lemma}
		Let $c_{\ell}$ be the probability that a GSAW on $\Delta^{\ast}_{\infty}$ starting at the wide corner has $C_{n}$ as its first segment where $n+2 = \ell$.
		Then the ordinary generating function $C(x)$ of $c_{\ell}$ is
		\begin{displaymath}
			C(x) = \frac{x^{2}}{9-3x}.
		\end{displaymath}
	\end{lemma}
	
	\begin{proof}
		The proof is the same as in the upward twist case.
		The only different is that the last step is a $N_{u}$ as opposed to an $E$ step, the probability of the step is the same, and the length doesn't change.
	\end{proof}
	
	\begin{definition}
		We define an upwards hook path of order $n$ on $\Delta^{\ast}_{\infty}$ starting at the wide corner to be the path
		\begin{displaymath}
			H^{u}_{n} = \underbrace{E\cdots E}_{\text{$n$ steps}}S_{u}\underbrace{W\cdots W}_{\text{$n+1$ steps}}.
		\end{displaymath}
	\end{definition}

	The probability of a GSAW on $\Delta^{\ast}_{\infty}$ starting at the wide corner having $H^{u}_{n}$ as its first segment is
	\begin{displaymath}
		\frac{1}{3^{n+2}}.
	\end{displaymath}
	The length of $H^{u}_{n}$ is $2n+2$.
	
	\begin{lemma}
		Let $h^{u}_{\ell}$ be the probability that a GSAW on $\Delta^{\ast}_{\infty}$ starting at the wide corner has $H^{u}_{n}$ as its first segment where $2n + 2 = \ell$.
		Then the ordinary generating function $H^{u}(x)$ of $h^{u}_{\ell}$ is
		\begin{displaymath}
			H^{u}(x) = \frac{x^{2}}{9-3x^{2}}.
		\end{displaymath}
	\end{lemma}
	
	\begin{proof}
		
		In order to create $H^{u}_{n}$ on $\Delta^{\ast}_{\infty}$, the path must move $n$ $E$ steps each with probability $1/3$, then a $S_{u}$ step with probability $1/3$ and then a $W$ step with probability $1/3$.
		The remaining $W$ steps are forced to exist, since there are no other options.
		Thus the probability of a GSAW on $\Delta^{\ast}_{\infty}$ starting at the wide corner having $H^{u}_{n}$ as its first segment is $1/(3^{n+2})$.
		
		Thus
		\begin{align*}
			H^{u}(x) & = \sumover{\ell=0}{\infty}{h^{u}_{\ell}x^{\ell}} 
			          = \sumover{\ell=2}{\infty}{h^{u}_{\ell}x^{\ell}} 
			          = \sumover{n=0}{\infty}{h^{u}_{2n+2}x^{2n+2}} \\
			         & = \sumover{n=0}{\infty}{\frac{1}{3^{n+2}}x^{2n+2}} 
			          = \frac{x^{2}}{9}\sumover{n=0}{\infty}{\left(\frac{x^{2}}{3}\right)^{n}} 
			          = \frac{\frac{x^{2}}{9}}{1 - \frac{x^{2}}{3}} 
			          =\boxed{ \frac{x^{2}}{9-3x^{2}}}.
		\end{align*}
	\end{proof}
	
	\begin{definition}
		We define a downward twist path of order $n$ on $\Delta^{\ast}_{\infty}$ starting at the wide corner to be the path
		\begin{displaymath}
			T^{d}_{n} = \underbrace{E\cdots E}_{\text{$n+1$ steps}}S_{d}E.
		\end{displaymath}
	\end{definition}

	The probability of a GSAW on $\Delta^{\ast}_{\infty}$ starting at the wide corner having $T^{d}_{n}$ as its first segment is
	\begin{displaymath}
		\frac{1}{2\cdot 3^{n+2}}.
	\end{displaymath}
	The length of $T^{d}_{n}$ is $n+3$.
	
	\begin{lemma}
		Let $t^{d}_{\ell}$ be the probability that a GSAW on $\Delta^{\ast}_{\infty}$ starting at the wide corner has $T^{d}_{n}$ as its first segment where $n+3 = \ell$.
		Then the ordinary generation function $T^{d}(x)$ of $t^{d}_{\ell}$ is
		\begin{displaymath}
			T^{d}(x) = \frac{x^{3}}{18-6x}.
		\end{displaymath}
	\end{lemma}
	
	\begin{proof}
		
		In order to create $T^{d}_{n}$ on $\Delta^{\ast}_{\infty}$, the path must move $n+1$ $E$ steps each with probability $1/3$, then a $S_{d}$ step with probability $1/3$ and then an $E$ step with probability $1/2$.
		Thus the probability of a GSAW on $\Delta^{\ast}_{\infty}$ starting at the wide corner having $T^{d}_{n}$ as its first segment is $1/(2\cdot 3^{n+2})$.
		
		Thus
		\begin{align*}
			T^{d}(x) & = \sumover{\ell=0}{\infty}{t^{d}_{\ell}x^{\ell}} 
			      = \sumover{\ell=3}{\infty}{t^{d}_{\ell}x^{\ell}} 
			          = \sumover{n=0}{\infty}{t^{d}_{n+3}x^{n+3}} \\
			         & = \sumover{n=0}{\infty}{\frac{1}{2\cdot 3^{n+2}}x^{n+3}} 
			          = \frac{x^{3}}{18}\sumover{n=0}{\infty}{\left(\frac{x}{3}\right)^{n}} 
			          = \frac{\frac{x^{3}}{18}}{1 - \frac{x}{3}} 
			         = \boxed{\frac{x^{3}}{18-6x}}.
		\end{align*}
	\end{proof}
	
	\begin{definition}
		We define a downward hook path of order $n$ on $\Delta^{\ast}_{\infty}$ starting at the wide corner to be the path
		\begin{displaymath}
			H^{d}_{n} = \underbrace{E\cdots E}_{\text{$n+1$ steps}}S_{d}\underbrace{W\cdots W}_{\text{$n+1$ steps}}.
		\end{displaymath}
	\end{definition}

	The probability of a GSAW on $\Delta^{\ast}_{\infty}$ starting at the wide corner having $H^{d}_{n}$ as its first segment is
	\begin{displaymath}
		\frac{1}{2\cdot 3^{n+2}}.
	\end{displaymath}
	The length of $H^{d}_{n}$ is $2n+3$.
	
	\begin{lemma}
		Let $h^{d}_{\ell}$ be the probability that a GSAW on $\Delta^{\ast}_{\infty}$ starting at the wide corner has $H^{d}_{n}$ as its first segment where $2n+3 = \ell$.
		Then the ordinary generation function $H^{d}(x)$ of $h^{d}_{\ell}$ is
		\begin{displaymath}
			H^{d}(x) = \frac{x^{3}}{18-6x^{2}}.
		\end{displaymath}
	\end{lemma}
	
	\begin{proof}
		
		In order to create $H^{d}_{n}$ on $\Delta^{\ast}_{\infty}$, the path must move $n+1$ $E$ steps each with probability $1/3$, then a $S_{d}$ step with probability $1/3$ and then an $E$ step with probability $1/2$.
		The remaining $W$ steps are forced to exist, since there are no other options.
		Thus the probability of a GSAW on $\Delta^{\ast}_{\infty}$ starting at the wide corner having $H^{d}_{n}$ as its first segment is $1/(2\cdot 3^{n+2})$.
		
		Thus
		\begin{align*}
			H^{d}(x) & = \sumover{\ell=0}{\infty}{h^{d}_{\ell}x^{\ell}} 
			          = \sumover{\ell=3}{\infty}{h^{d}_{\ell}x^{\ell}} 
			          = \sumover{n=0}{\infty}{h^{d}_{2n+3}x^{2n+3}} \\
			         & = \sumover{n=0}{\infty}{\frac{1}{2\cdot 3^{n+2}}x^{2n+3}} 
			         = \frac{x^{3}}{18}\sumover{n=0}{\infty}{\left(\frac{x^{2}}{3}\right)^{n}} 
			          = \frac{\frac{x^{3}}{18}}{1 - \frac{x^{2}}{3}} 
			          = \boxed{\frac{x^{3}}{18-6x^{2}}}.
		\end{align*}
		
	\end{proof}
	
	\subsubsection{Solving the Triangular Case}
	
	\begin{definition}
	
		Let $w_{\ell}$ be the probability that a GSAW on $\Delta^{\ast}_{\infty}$ starting at the wide corner terminates with length $\ell$.
		Let $n_{\ell}$ be the probability that a GSAW on $\Delta^{\ast}_{\infty}$ starting at the narrow corner terminates with length $\ell$.
		
		Let $N(x)$ be the ordinary generation function of $n_{\ell}$ and $W(x)$ be the ordinary generation function of $w_{\ell}$.
	
	\end{definition}
	
	\begin{lemma}\label{lemma:narrow-corner}
		It is the case that
		\begin{displaymath}
			N(x) = \frac{x}{2}N(x) + \frac{x}{2}W(x).
		\end{displaymath}
	\end{lemma}
	
	\begin{proof}
		Consider a GSAW on $\Delta^{\ast}_{\infty}$ that starts at the narrow corner.
		Either it move $E$ with probability $1/2$ or it moves $N_{u}$ with probability $1/2$.
		
		If it moves $E$, then the rest of the path is a GSAW on $\Delta^{\ast}_{\infty}$ starting at the wide corner by symmetry.
		If it moves $N_{u}$, then the rest of the path is a GSAW on $\Delta^{\ast}_{\infty}$ starting at the wide corner by symmetry.
		Thus
		\begin{displaymath}
			N(x) = \frac{x}{2}N(x) + \frac{x}{2}W(x).
		\end{displaymath}
	\end{proof}
	
	\begin{lemma}\label{lemma:wide-corner}
		It is the case that
		\begin{displaymath}
			W(x) = H^{d}(x) + H^{u}(x) + T^{u}(x)W(x) + T^{d}(x)N(x) + \frac{x^{2}}{3}N(x).
		\end{displaymath}
	\end{lemma}
	
	\begin{proof}
		Consider a GSAW on $\Delta^{\ast}_{\infty}$ that starts at the wide corner.
		The first step must either be a $N_{d}$ step, a $N_{u}$ step for an $E$ step.
		With $1/3$ probability, the path takes a $N_{d}$ step, which must immediately be followed by an $E$ step since it is the only option.
		In this case two steps have been taken and the rest of the path is a GSAW on $\Delta^{\ast}_{\infty}$ starting at the narrow corner.
		
		If the first step is a $N_{u}$ step or $E$ step, then the GSAW must start with either a $H^{u}_{n},H^{d}_{n},T^{u}_{n}$,$T^{d}_{n}$ or $C_{n}$ path.
		If it is a $H^{d_{n}}$ or $H^{u}_{n}$ path, the GSAW terminates.
		If it is a $T^{u}_{n}$ path then the rest of the path is a GSAW on $\Delta^{\ast}_{\infty}$ starting at the wide corner.
		If it is a $T^{d}_{n}$ path then the rest of the path is a GSAW on $\Delta^{\ast}_{\infty}$ starting at the narrow corner.
		If it is a $C_{n}$ path then the rest of the path is a GSAW on $\Delta^{\ast}_{\infty}$ starting at the narrow corner.
		
		Thus
		\begin{displaymath}
			W(x) = H^{d}(x) + H^{u}(x) + T^{u}(x)W(x) + T^{d}(x)N(x) + C(x)N(x) + \frac{x^{2}}{3}N(x).
		\end{displaymath}
	\end{proof}
	
	\begin{theorem}
		\begin{displaymath}
			N(x) = \frac{x^{3}(x+2)(x-3)}{(x^{2}-3)(x^{4} - 6x^{3} + 2x^{2}-30x + 36)}
		\end{displaymath}
		and
		\begin{displaymath}
			W(x) = \frac{x^{2}(x+2)(x-3)(x-2)}{(3-x^{2})(x^{4} - 6x^{3} + 2x^{2}-30x + 36)}
		\end{displaymath}
	\end{theorem}
	
	\begin{proof}
		From lemma~\ref{lemma:narrow-corner} we can solve for $N(x)$ to get
		\begin{displaymath}
			N(x) = \frac{xW(x)}{2-x}
		\end{displaymath}
		
		Thus the equation from lemma~\ref{lemma:wide-corner} can be written as
		\begin{displaymath}
			W(x) = H^{d}(x) + H^{u}(x) + T^{u}(x)W(x) + \frac{xW(x)T^{d}(x)}{2-x} + \frac{xC(x)W(x)}{2-x} + \frac{x^{3}W(x)}{3(2-x)}.
		\end{displaymath}
		
		Solving for $W(x)$ we get that
		\begin{displaymath}
			W(x) = \frac{H^{d}(x) + H^{u}(x)}{1-T^{u}(x)-\frac{xT^{d}(x)}{2-x} - \frac{xC(x)}{2-x} - \frac{x^{3}}{3(2-x)}}.
		\end{displaymath}
		Substituting in the formulas for the known generation functions we get that
		\begin{displaymath}
			W(x) = \frac{x^{2}(x+2)(x-3)(x-2)}{(3-x^{2})(x^{4} - 6x^{3} + 2x^{2}-30x + 36)}
		\end{displaymath}
		and so
		\begin{displaymath}
			N(x) = \frac{x^{3}(x+2)(x-3)}{(x^{2}-3)(x^{4} - 6x^{3} + 2x^{2}-30x + 36)}
		\end{displaymath}
	\end{proof}
	
	\begin{corollary}
		The probability that a GSAW on $\Delta^{\ast}_{\infty}$ starting at either the wide or narrow corner terminates is 1.
	\end{corollary}
	
	\begin{corollary}
		The expected number of steps taken before a GSAW on $\Delta^{\ast}_{\infty}$ starting at the wide corner terminates is $91/6$ with variance $793/4$.
		The expected number of steps taken before a GSAW on $\Delta^{\ast}_{\infty}$ starting at the narrow corner terminates is $103/6$ with variance $801/4$.
	\end{corollary}
	
	\subsubsection{The double sided case}
	
	\begin{definition}
		Let $\delta_{\ell}$ be the probability that a GSAW on $\Delta_{\infty}$ terminates with length $\ell$.
		
		Let $\Delta(x)$ be the ordinary generation function of $\delta_{\ell}$.
	\end{definition}
	
	 We observe that by symmetry, every starting point on $\Delta_{\infty}$ is the same.
	
	\begin{theorem}\label{theorem:trapping-length-triangle-lattice}
		\begin{displaymath}
			\Delta(x) = \frac{x^{4}(2+x)(x^{4}+x^{3} - 6x - 12)}{4(x^{2}-3)^{2}(x^{4} - 6x^{3} + 2x^{2} - 30x + 36)}
		\end{displaymath}
	\end{theorem}
	
	\begin{proof}
		Without loss of generality assume we start at $(0,0)$ on $\Delta_{\infty}$.
		Then the first step must be one of $W,E,N_{d}$, and $N_{u}$.
		By symmetry we may assume that it is either a $N_{u}$ or $E$ step.
		
		If the first step of the path is a $N_{u}$, then the next step will prevent the path from looping back and cause the rest of the path is be a GSAW on $\Delta_{\infty}$
		If the second step is an $E$ step then the rest of the path is a GSAW on $\Delta_{\infty}$ starting at the wide corner.
		This has a $1/3$ probability of happening given that the first step was $N_{u}$.
		Otherwise the second step is either a $W$ or $S_{u}$ step and the rest of the path is a GSAW on $\Delta_{\infty}$ starting at the narrow corner.
		
		If the first step of the path is an $E$ step, then one of the following must happen.
		\begin{enumerate}
		
			\item
			
			    The Path Hooks Back:
			
				The path makes a $H^{u}_{n}$ or $H^{d}_{n}$ path.
				Since these paths only visits the vertex that can access the other side at the end, the generation functions are the same.
				A extra $W$ step is taken, at which point the rest of the path is a GSAW on $\Delta_{\infty}$ starting at the narrow corner.
			
			\item
			
			    The Path Twists:
			
				The path makes a $T^{u}_{n}$ or $T^{d}_{n}$ path.
				Since these paths never visits the vertex that can access the other side at the end, the generation functions are the same.
				The rest of the path is a GSAW on $\Delta_{\infty}$ starting at the corresponding corner.
			
			\item
			
			    The Path Becomes Crooked:
			
				The path makes a $C_{n}$ path at which point the rest of the path is a GSAW on $\Delta_{\infty}$ starting at the narrow corner.
			
			\item
			
			    The Path Takes a $N_{d}$ step:
			
				If the path takes a $N_{d}$ step, then it can either take a $W$ or $E$ step.
				However, in both cases the rest of the path is a GSAW on $\Delta_{\infty}$ starting at the narrow corner.
		\end{enumerate}
		
		Thus we have that
		\begin{align*}
			S(x) = & \frac{x}{2}\left(\frac{2x}{3}N(x) + \frac{x}{3}W(x)\right) \\
			       & + \frac{x}{2}\left[\left(H^{d}(x) + H^{u}(x)\right)xN(x) + T^{u}(x)W(x) + T^{d}(x)N(x) + C(x)N(x) + \frac{x^{2}}{3}N(x)\right]
		\end{align*}
		Since the other functions are known, substituting them in gives
		\begin{displaymath}
			S(x) = \frac{x^{4}(2+x)(x^{4}+x^{3} - 6x - 12)}{4(x^{2}-3)^{2}(x^{4} - 6x^{3} + 2x^{2} - 30x + 36)}.
		\end{displaymath}
	\end{proof}
	
	\begin{corollary}
		The probability that a GSAW on $\Delta_{\infty}$ terminates is 1.
	\end{corollary}
	
	\begin{corollary}
		The expected number of steps taken before a GSAW on $\Delta_{\infty}$ starting at the wide corner terminates is $941/48\approx19.6$ with variance $51919/256 \approx 202.8$ (standard deviation $\approx 14.2$).
	\end{corollary}
	
\subsection{Biased Square Ladders}

\subsubsection{Definitions}

	\begin{definition}
		Let $G = (V,E)$ be a graph.
		Then a nearest-neighbor self-attraction GSAW is a GSAW where each step is as follows.
		Let $P = (v_{1},v_{2},\cdots,v_{n})$ be a path on the graph G.
		Let $M = N(v_{n})\setminus P$ be the set of all neighbors of $v_{n}$ that are not in the path.
		
		For any vertex $w$ in $G$ not already on the path, its energy is $E(w) = C^{\abs{N(w)\cap (P \setminus v_{n})}}$, where $C$ is a positive constant.
		Thus the energy of a vertex is a constant raised to the power of the number of neighbors already in the path.
		
		The probability of $P$ moving to the neighboring vertex $w^{\ast}$ is
		\begin{displaymath}
			\frac{E(w^{\ast})}{\sum\limits_{w \in M}E(w)}
		\end{displaymath}
	\end{definition}
	
	\begin{figure}
		\begin{center}
			\begin{tikzpicture}[scale=1]
				\draw[help lines] (0,-1) grid (2,1);
				\filldraw (.125,.125) -- (-.125,.125) -- (-.125,-.125) -- (.125,-.125) -- cycle;
				\filldraw (1,0) ++(0,.125) -- ++(-.125,-.125) -- ++(.125,-.125) -- ++(.125,.125) -- ++(-.125,.125);
				\filldraw (0,-1) circle (.1);
				\filldraw (1,-1) circle (.1);
				\filldraw (2,-1) circle (.1);
				\filldraw (1,0) circle (.1);
				\filldraw (2,0) circle (.1);
				\filldraw (0,1) circle (.1);
				\filldraw (1,1) circle (.1);
				\filldraw (2,1) circle (.1);
				\node[anchor=south east] at (0,0) {$v_{1}$};
				\node[anchor=south east] at (1,0) {$v_{2}$};
				\node[anchor=south east] at (2,0) {$w_{1}$};
				\node[anchor=south east] at (1,-1) {$w_{2}$};
				\node[anchor=south east] at (1,1) {$w_{3}$};
				\draw[ultra thick,postaction={on each segment={mid arrow=black}}] (0,0) -- (1,0);
			\end{tikzpicture}
		\end{center}
		\caption{Since $w_{1},w_{2},$ and $w_{3}$ have no neighbors that belong to $P$ that are not $v_{2}$ their energy is $C^{0} = 1$, and so the probability of visiting any one of them is $1/3$.}
		\label{figure:nearest-neighbor-example-1}
	\end{figure}
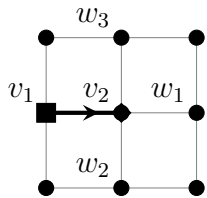
	
	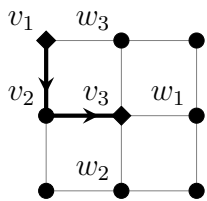
\begin{figure}
		\begin{center}
			\begin{tikzpicture}[scale=1]
				\draw[help lines] (0,-1) grid (2,1);
				\filldraw (0,1) ++ (.125,0) -- ++(-.125,.125) -- ++(-.125,-.125) -- ++(.125,-.125) -- cycle;
				\filldraw (1,0) ++(0,.125) -- ++(-.125,-.125) -- ++(.125,-.125) -- ++(.125,.125) -- ++(-.125,.125);
				\filldraw (0,-1) circle (.1);
				\filldraw (1,-1) circle (.1);
				\filldraw (2,-1) circle (.1);
				\filldraw (0,0) circle (.1);
				\filldraw (2,0) circle (.1);
				\filldraw (1,1) circle (.1);
				\filldraw (2,1) circle (.1);
				\node[anchor=south east] at (0,1) {$v_{1}$};
				\node[anchor=south east] at (0,0) {$v_{2}$};
				\node[anchor=south east] at (1,0) {$v_{3}$};
				\node[anchor=south east] at (2,0) {$w_{1}$};
				\node[anchor=south east] at (1,-1) {$w_{2}$};
				\node[anchor=south east] at (1,1) {$w_{3}$};
				\draw[ultra thick,postaction={on each segment={mid arrow=black}}] (0,1) -- (0,0) -- (1,0);
			\end{tikzpicture}
		\end{center}
		\caption{Since $w_{1}$ and $w_{2}$ have no neighbors that belong to $P$ that are not $v_{3}$ their energy is $C^{0} = 1$.
		The vertex $w_{3}$ has one neighbor that belongs to $P$ that is not $v_{3}$ and so its energy is $C^{1} = C$.
		Thus the probabilities of $w_{1}$ or $w_{2}$ being the next vertex is $\frac{1}{2+C}$ and the probability of $w_{3}$ being the next vertex are each $\frac{C}{2+C}$}
		\label{figure:nearest-neighbor-example-2}
	\end{figure}

\subsubsection{Precursor Functions}

	\begin{lemma}
		Let $h^{\ast}_{\ell,w}$ be the probability that a GSAW with nearest-neighbor weighting on $L^{\ast}_{\infty}$ terminates as a hook path of length $\ell$ and width $w$.
		Then the bivariate ordinary generation function $H^{\ast}(x,y)$ of $h^{\ast}_{\ell,w}$ is
		\begin{displaymath}
			H^{\ast}(x,y) = \frac{C}{1+C} \cdot \frac{x^{3}y}{4-2x^{2}y}
		\end{displaymath}
	\end{lemma}
	
	\begin{proof}
		Since we are only considering hook paths that terminate, $H_{0}$ is not considered.
		
		The first $n+1$ steps all have probability $1/2$ of happening since their neighbors are adjacent to the path.
		However after moving $N$, the $E$ step has no neighbor in the path and the $W$ step as one neighbor in the path.
		Thus the probability of taking the first $W$ step is $C/(1+C)$.
		
		Thus for the remaining hook paths $H^{\ast}_{n}$, each as length $2n+1$ and probability $(1/2)^{n+1}\frac{C}{1+C}$.
		
		Thus
		\begin{align*}
			H^{\ast}(x,y) & = \sumover{\ell=0}{\infty}{\sumover{w=0}{\infty}{h^{\ast}_{\ell,w}x^{\ell}y^{w}}} 
			               = \sumover{w=0}{\infty}{h^{\ast}_{2w+1,w}x^{2w+1}y^{w}} 
			               = \sumover{k=1}{\infty}{h^{\ast}_{2k+1,k}x^{2k+1}y^{k}} \\
			              & = \sumover{k=1}{\infty}{\left(\frac{1}{2}\right)^{k+1}\frac{C}{1+C} x^{2k+1}y^{k}} 
			               = \sumover{k=0}{\infty}{\left(\frac{1}{2}\right)^{k+2}\left(\frac{C}{1+C}\right)x^{2k+3}y^{k+1}} \\
			               &= \frac{x^{3}y}{4}\left(\frac{C}{1+C}\right)\sumover{k=0}{\infty}{\left(\frac{x^{2}y}{2}\right)^{k}} 
			               =\frac{C}{1+C} \cdot \frac{\frac{x^{3}y}{4}}{1 - \frac{x^{2}y}{2}} 
			               = \boxed{\frac{C}{1+C} \cdot \frac{x^{3}y}{4-2x^{2}y}}.
		\end{align*}
		
	\end{proof}
	
	\begin{lemma}
		Let $t^{\ast}_{\ell,w}$ be the probability that the first $\ell$ steps of a GSAW with nearest-neighbor weighing on $L^{\ast}_{\infty}$ is a twist path of length $\ell$ and width $w$.
		Then the bivariate ordinary generating function $T^{\ast}(x,y)$ of $t^{\ast}_{\ell,w}$ is
		\begin{displaymath}
			T^{\ast}(x,y) = \frac{x^{2}y}{2} + \frac{1}{1+C}\frac{x^{3}y^{2}}{4-2xy}.
		\end{displaymath}
	\end{lemma}
	
	\begin{proof}
		We observe that the minimum length of a twist path is 2, and that for any twist path $\ell = 1 + w$.
		If the twist path is of length two, then it must first move $N$ with probability $1/2$ and then it has no options but to move $E$.
		
		Otherwise it first makes an $E$ step with probability $1/2$ and continues to make $E$ steps with  the same probability, until it takes a $N$ step with probability $1/2$.
		Since there now exist points to the left, the probability of the last $E$ step is $1/(1+C)$.
		Thus
		\begin{align*}
			T^{\ast}(x,y) & = \sumover{\ell=0}{\infty}{\sumover{w=0}{\infty}{t^{\ast}_{\ell,w}x^{\ell}y^{w}}} 
			        = \sumover{\ell=2}{\infty}{t^{\ast}_{\ell,\ell-1}x^{\ell}y^{\ell-1}} 
			       = t^{\ast}_{2,1}x^{2}y + \sumover{\ell=3}{\infty}{t_{\ell,\ell-1}x^{\ell}y^{\ell-1}} \\
			       & = t^{\ast}_{2,1}x^{2}y + \sumover{k=0}{\infty}{t_{k+3,k+2}x^{k+3}y^{k+2}} 
			        = \frac{x^{2}y}{2} + \sumover{k=0}{\infty}{\frac{1}{1+C}\left(\frac{1}{2}\right)^{k+2}x^{k+3}y^{k+2}} \\
			       & = \frac{x^{2}y}{2} + \frac{x^{3}y^{2}}{(1+C)4}\sumover{k=0}{\infty}{\left(\frac{xy}{2}\right)^{k}} 
			        = \frac{x^{2}y}{2} + \frac{\left(\frac{x^{3}y^{2}}{(1+C)4}\right)}{1 - \frac{xy}{2}} 
			        =\boxed{ \frac{x^{2}y}{2} + \frac{x^{3}y^{2}}{(1+C)(4-2xy)}}.
		\end{align*}

	\end{proof}
	
	\begin{figure}
		\begin{center}
			\begin{tikzpicture}[scale=1]
				\fill[black!50!white] (-.25 ,-.25) rectangle (0.5,1.25);
				\draw[very thick] (-.25 ,-.25) rectangle (0.5,1.25);
				\draw[help lines,black] (0,0) grid (4,1);
				\draw[help lines,black] (4,1) -- (4.25,1);
				\draw[help lines,black] (4,0) -- (4.25,0);
				\draw (4.5,0.5) node {$\scriptstyle\cdots\,$};
				\filldraw (1.125,.125) -- (.875,.125) -- (.875,-.125) -- (1.125,-.125) -- cycle;
				\filldraw (0,0) circle (.1);
				\filldraw (2,0) circle (.1);
				\filldraw (3,0) circle (.1);
				\filldraw (4,0) circle (.1);
				\filldraw (0,1) circle (.1);
				\filldraw (1,1) circle (.1);
				\filldraw (2,1) circle (.1);
				\filldraw (3,1) circle (.1);
				\filldraw (4,1) circle (.1);
			\end{tikzpicture}
		\end{center}
		\caption{In this case, the path starts in the same place, but there exists two point to the left which are considered to be ``part of the path'' for the purpose of nearest neighbor calculations. A GSAW on such a set-up will be called a GSAW on $\underline{L}^{\ast}_{\infty}$}
		\label{figure:one-sided-ladder-base-weight}
	\end{figure}
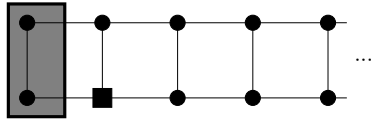
	
	\begin{lemma}
		Let $\underline{h}^{\ast}_{\ell,w}$ be the probability that a GSAW with nearest-neighbor weighing on $\underline{L}^{\ast}_{\infty}$ terminates as a hook path of length $\ell$ and width $w$.
		Then the bivariate ordinary generating function $\underline{H}^{\ast}(x,y)$ of $\underline{h}^{\ast}_{\ell,w}$ is
		\begin{displaymath}
			\underline{H}^{\ast}(x,y) = \frac{C^{2}x^{3}y}{2(1+C)(1+C^{2})} + \frac{Cx^{5}y^{2}}{2(1+C)^{2}(2-x^{2}y)}
		\end{displaymath}
	\end{lemma}
	
	\begin{proof}
		Since we are only considering hook paths that terminate, $\underline{H}^{\ast}_{0}$ is not considered.
		
		The first step has probability $1/(1+C)$ probability of happening.
		The next $\ell$ steps all have probability $1/2$ of happening since their neighbors are not adjacent to the path.
		However after moving $N$, the $E$ step has no neighbor in the path and the $W$ step has either one or two neighbors in the path.
		If $\ell = 3$, then this step has two neighbors and has probability $C^{2}/(1+C^{2})$ of happening, otherwise it has probability $C/(1+C)$ of happening.
		The rest of the steps are then determined.
		
		Thus
	\begin{align*}
			\underline{H}^{\ast}(x,y) & = \sumover{\ell=0}{\infty}{\sumover{w=0}{\infty}{\underline{h}^{\ast}_{\ell,w}x^{\ell}y^{w}}} = \sumover{w=0}{\infty}{\underline{h}^{\ast}_{2w+1,w}x^{2w+1}y^{w}} 
			 = \underline{h}^{\ast}_{3,1}x^{3}y + \sumover{w=2}{\infty}{\underline{h}^{\ast}_{2w+1,w}x^{2w+1}y^{w}} \\
			& = \frac{C^{2}x^{3}y}{2(1+C)(1+C^{2})} + \sumover{k=0}{\infty}{\underline{h}^{\ast}_{2k+1,k}x^{2k+5}y^{k+2}} \\
			& = \frac{C^{2}x^{3}y}{2(1+C)(1+C^{2})} + \sumover{k=0}{\infty}{\frac{C}{2^{k+2}(1+C)^{2}}x^{2k+5}y^{k+2}} \\
			& = \frac{C^{2}x^{3}y}{2(1+C)(1+C^{2})} + \frac{Cx^{5}y^{2}}{4(1+C)^{2}}\sumover{k=0}{\infty}{\left(\frac{x^{2}y}{2}\right)^{k}} \\
			& = \frac{C^{2}x^{3}y}{2(1+C)(1+C^{2})} + \frac{Cx^{5}y^{2}}{4(1+C)^{2}}\frac{1}{1-\frac{x^{2}y}{2}} \\
			& =\boxed{ \frac{C^{2}x^{3}y}{2(1+C)(1+C^{2})} + \frac{Cx^{5}y^{2}}{2(1+C)^{2}(2-x^{2}y)}}
		\end{align*}
	\end{proof}
	
	\begin{lemma}
		Let $\underline{t}_{\ell,w}$ be the probability that the first $\ell$ steps of a GSAW with nearest-neighbor weighing on $\underline{L}^{\ast}_{\infty}$ is a twist path of length $\ell$ and width $w$.
		Then the bivariate ordinary generating function $\underline{T}^{\ast}(x,y)$ of $\underline{t}_{\ell,w}$ is
		\begin{displaymath}
			\underline{T}^{\ast}(x,y) = \frac{Cx^{2}y}{C+1} + \frac{x^{3}y^{2}}{2(1+C)(1+C^{2})} + \frac{x^{4}y^{3}}{2(1+C)^{2}(2-xy)}.
		\end{displaymath}
	\end{lemma}
	
	\begin{proof}
		We observe that the minimum length of a twist path is 2, and that for any twist path $\ell = 1 + w$.
		
		If the twist path is of length two, then it must first move $N$ with probability $C/(C+1)$ and then it has no other options but to move $E$.
		
		Otherwise it first makes an $E$ step with probability $1/(C+1)$.
		From here there are two options, it moves $N$ with probability 1/2, but then must move $E$ with probability $1/(1+C^{2})$ since the vertex west of it is adjacent to the path and the floor.
		This is the only twist path of length two.
		
		If this doesn't happen, then it continues to move east at which point when it does move north there is only one vertex adjacent in the backwards direction.
		Thus the final $E$ step will happen with probability $1/(C+1)$ and continues to make $E$ steps with  the same probability, until it takes a $N$ step with probability $1/2$.
		Since there now exist points to the left, the probability of the last $E$ step is $1/(1+C)$.
		Thus
		\begin{align*}
			\underline{T}^{\ast}(x,y) & = \sumover{\ell=0}{\infty}{\sumover{w=0}{\infty}{\underline{t}^{\ast}_{\ell,w}x^{\ell}y^{w}}} 
			        = \sumover{\ell=2}{\infty}{\underline{t}^{\ast}_{\ell,\ell-1}x^{\ell}y^{\ell-1}} 
			        = \underline{t}^{\ast}_{2,1}x^{2}y + \underline{t}^{\ast}_{3,2}x^{3}y^{2} + \sumover{\ell=3}{\infty}{\underline{t}_{\ell,\ell-1}x^{\ell}y^{\ell-1}} \\
			       & = \frac{Cx^{2}y}{C+1} + \frac{x^{3}y^{2}}{2(C+1)(C^{2}+1)} + \sumover{k=0}{\infty}{\frac{x^{k+3}y^{k+2}}{2^{k+1}(C+1)^{2}}} \\
			       & = \frac{Cx^{2}y}{C+1} + \frac{x^{3}y^{2}}{2(C+1)(C^{2}+1)} + \frac{x^{3}y^{2}}{2}\sumover{k=0}{\infty}{\left(\frac{xy}{2}\right)^{k}} \\
			       & = \frac{Cx^{2}y}{C+1} + \frac{x^{3}y^{2}}{2(C+1)(C^{2}+1)} + \frac{x^{3}y^{2}}{2}\frac{1}{1-\frac{xy}{2}} \\
			       & =\boxed{ \frac{x^{3}y^{2}}{2(1+C)(1+C^{2})} + \frac{x^{4}y^{3}}{2(1+C)^{2}(2-xy)}}.
		\end{align*}
	\end{proof}

\subsubsection{The One Sided Case}

	\begin{theorem}
		Let $\underline{p}^{\ast}_{\ell,w}$ be the probability that a GSAW with nearest-neighbor weighing on $\underline{L}^{\ast}_{\infty}$ terminates having taking $\ell$ steps and a width of $w$.	
		Then the bivariate ordinary generating function $\underline{L}^{\ast}_{\infty}(x,y)$ of $\underline{p}^{\ast}_{\ell,w}$ is
		\begin{displaymath}
			\underline{L}^{\ast}_{\infty}(x,y) = \frac{\underline{H}^{\ast}(x,y)}{1-\underline{T}^{\ast}(x,y)}.
		\end{displaymath}
	\end{theorem}
	
	\begin{proof}
		We first consider an arbitrary GSAW with nearest-neighbor weighing on $\underline{L}^{\ast}_{\infty}$.
		We observe that with probability 1 the GSAW makes a $N$ step at some point since the only way to avoid such a step is an infinite series of $E$ steps which has probability 0 of happening.
		After the GSAW makes the $N$ step it must either take an $E$ step or $W$ step.
		If it takes a $W$ step then it forms a hook path and terminates.
		It is takes an $E$ step then it forms a twist path and must continue.
		However, at the end of the twist path, the remaining spaces reachable from the current position is equivalent to starting a new GSAW with nearest-neighbor weighing on $\underline{L}^{\ast}_{\infty}$.
		
		Thus every GSAW with nearest-neighbor weighing on $\underline{L}^{\ast}_{\infty}$ is either a hook path or a twist path followed by another GSAW with nearest-neighbor weighing on $\underline{L}^{\ast}_{\infty}$.

		Thus
		\begin{displaymath}
			\underline{L}^{\ast}_{\infty}(x,y) = \underline{T}^{\ast}(x,y)\underline{L}^{\ast}_{\infty}(x,y) + \underline{H}^{\ast}(x,y).
		\end{displaymath}
		Solving for $\underline{L}^{\ast}_{\infty}(x,y)$ we get
		\begin{displaymath}
			\underline{L}^{\ast}_{\infty}(x,y) = \frac{\underline{H}^{\ast}(x,y)}{1-\underline{T}^{\ast}(x,y)}.
		\end{displaymath}
	\end{proof}
	
	\begin{theorem}
		Let $p^{\ast}_{\ell,w}$ be the probability that a GSAW with nearest-neighbor weighing on $L^{\ast}_{\infty}$ terminates having taking $\ell$ steps and a width of $w$.	
		Then the bivariate ordinary generating function $L^{\ast}_{\infty}(x,y)$ of $p^{\ast}_{\ell,w}$ is
		\begin{displaymath}
			L^{\ast}_{\infty}(x,y) = T^{\ast}(x,y)\underline{L}^{\ast}_{\infty}(x,y) + H^{\ast}(x,y).
		\end{displaymath}
	\end{theorem}
	
	\begin{proof}
		We consider an arbitrary GSAW with nearest-neighbor weighing on $L^{\ast}_{\infty}$.
		Use similar logic as before we have that such a GSAW either makes a hook path and terminates or makes a twist path and then the rest is a GSAW with nearest-neighbor weighing on $\underline{L}^{\ast}_{\infty}$.
		
		Thus
		\begin{displaymath}
			L^{\ast}_{\infty}(x,y) = T^{\ast}(x,y)\underline{L}^{\ast}_{\infty}(x,y) + H^{\ast}(x,y).
		\end{displaymath}
	\end{proof}
	
	\begin{corollary}
		The probability that a GSAW with nearest-neighbor weighing on $L^{\ast}_{\infty}$ or $\underline{L}^{\ast}_{\infty}$ terminates is 1.
	\end{corollary}
	
	\begin{corollary}
		The expected number of steps taken before a GSAW with nearest-neighbor weighing on $L^{\ast}_{\infty}$ terminates is
		\begin{displaymath}
		    \frac{16+42 C+51 C^2+59 C^3+36 C^4+4 C^5}{2 C+4 C^2+6 C^3+4 C^4}
		\end{displaymath}
		with variance
		\begin{center}
		    \begin{displaymath}\Scale[1.2]{
			    \frac{256 + 1104 C + 2488 C^2 + 3998 C^3 + 4981 C^4 + 4864 C^5 + 3769 C^6 + 2292 C^7 + 1000 C^8 + 288 C^9 + 
 48 C^{10}}{4 C^2 (1 + C)^2 (1 + C + 2 C^2)^2}}.
		    \end{displaymath}
		\end{center}
		The expected width is
		\begin{displaymath}
		    \frac{C^4+7 C^3+8 C^2+6 C+6}{2 C^3+C^2+C}
		\end{displaymath}
		with variance
		\begin{displaymath}
			\frac{(C+1) \left(3 C^7+9 C^6+33 C^5+59 C^4+62 C^3+84 C^2+34 C+36\right)}{C^2 \left(2 C^2+C+1\right)^2}.
		\end{displaymath}
	\end{corollary}
	
	\begin{corollary}
		The expected number of steps taken before a GSAW with nearest-neighbor weighing on $\underline{L}^{\ast}_{\infty}$ terminates is
		\begin{displaymath}
		    \frac{4 C^4+14 C^3+12 C^2+14 C+8}{2 C^3+C^2+C}
		\end{displaymath} with variance
		\begin{displaymath}
			\frac{2 \left(8 C^8+24 C^7+72 C^6+119 C^5+159 C^4+169 C^3+127 C^2+74 C+32\right)}{C^2 \left(2 C^2+C+1\right)^2}.
		\end{displaymath}
		The expected width is
		\begin{displaymath}
		    \frac{2 C^4+6 C^3+7 C^2+7 C+6}{2 C^3+C^2+C}
		\end{displaymath}
		with variance
		\begin{displaymath}
			\frac{2 \left(2 C^8+6 C^7+21 C^6+45 C^5+60 C^4+74 C^3+59 C^2+35 C+18\right)}{C^2 \left(2 C^2+C+1\right)^2}.
		\end{displaymath}
	\end{corollary}

\subsubsection{The Double Sided Case}

	\begin{theorem}
		Let $H(x,y)$ be the bivariate ordinary generating function for paths for which the first segment is a hook path on $L_{\infty}$.
		Then
		\begin{displaymath}
			H(x,y) = \frac{x^{2}y}{3} + \frac{2Cx^{4}y^{2}}{3(1+C)(2-x^{2}y)}
		\end{displaymath}
	\end{theorem}
	
	\begin{theorem}
		Let $T(x,y)$ be the bivariate ordinary generating function for paths for which the first segment is a twist path on $L_{\infty}$.
		Then
		\begin{displaymath}
			T(x,y) = \frac{2x^{3}y^{2}}{3(1+C)(2-xy)}
		\end{displaymath}
	\end{theorem}
	
	\begin{theorem}
		Let $L_{\infty}(x,y)$ be the bivariate ordinary generating function for a GSAW that terminates.
		Then
		\begin{displaymath}
			L_{\infty}(x,y) = (H(x,y) + T(x,y))\underline{L}^{\ast}_{\infty}(x,y)
		\end{displaymath}
	\end{theorem}
	
	\begin{corollary}
		The probability that a GSAW with nearest-neighbor weighing on $L_{\infty}$ terminates is 1.
	\end{corollary}
	
	\begin{corollary}
		The expected number of steps taken before a GSAW with nearest-neighbor weighing on $L_{\infty}$ terminates is
		\begin{displaymath}
		    \frac{2 \left(6 C^5+41 C^4+56 C^3+51 C^2+38 C+12\right)}{3 C (C+1) \left(2 C^2+C+1\right)}
		\end{displaymath}
		with variance
		\begin{center}
		    \begin{displaymath}\Scale[1.2]{
		        \frac{2 \left(72 C^{10}+360 C^9+1312 C^8+2975 C^7+4693 C^6+5864 C^5+5822 C^4+4551 C^3+2773 C^2+1242 C+288\right)}{9 C^2 (C+1)^2 \left(2 C^2+C+1\right)^2}}.
		    \end{displaymath}
		\end{center}
		The expected width is
		\begin{displaymath}
		    \frac{6 C^4+32 C^3+28 C^2+28 C+18}{6 C^3+3 C^2+3 C}
		\end{displaymath}
		with variance
		\begin{displaymath}
		    \frac{2 \left(18 C^8+54 C^7+229 C^6+445 C^5+590 C^4+686 C^3+541 C^2+315 C+162\right)}{9 C^2 \left(2 C^2+C+1\right)^2}.
		\end{displaymath}
	\end{corollary}
	

\section{Discussion}

Figure \ref{fig:trap2} shows the trapping probability distributions in the square and triangular strips from the generating functions in Theorems~\ref{theorem:trapping-length-square-lattice} and~\ref{theorem:trapping-length-triangle-lattice}. Several features are apparent which can be compared to those observed in infinite lattices (Figure \ref{fig:trapinf}, Appendix). There is strong even-odd asymmetry for small trapping length in the square lattice, but not the triangular lattice. The triangular lattice has a higher mean trapping length than the square lattice, but the triangular probability distribution decays with a weaker exponential tail compared to the square lattice. All of these features are observed when comparing the infinite square and triangular lattices. 

\begin{figure}
    \centering
    \includegraphics[width=0.8\textwidth]{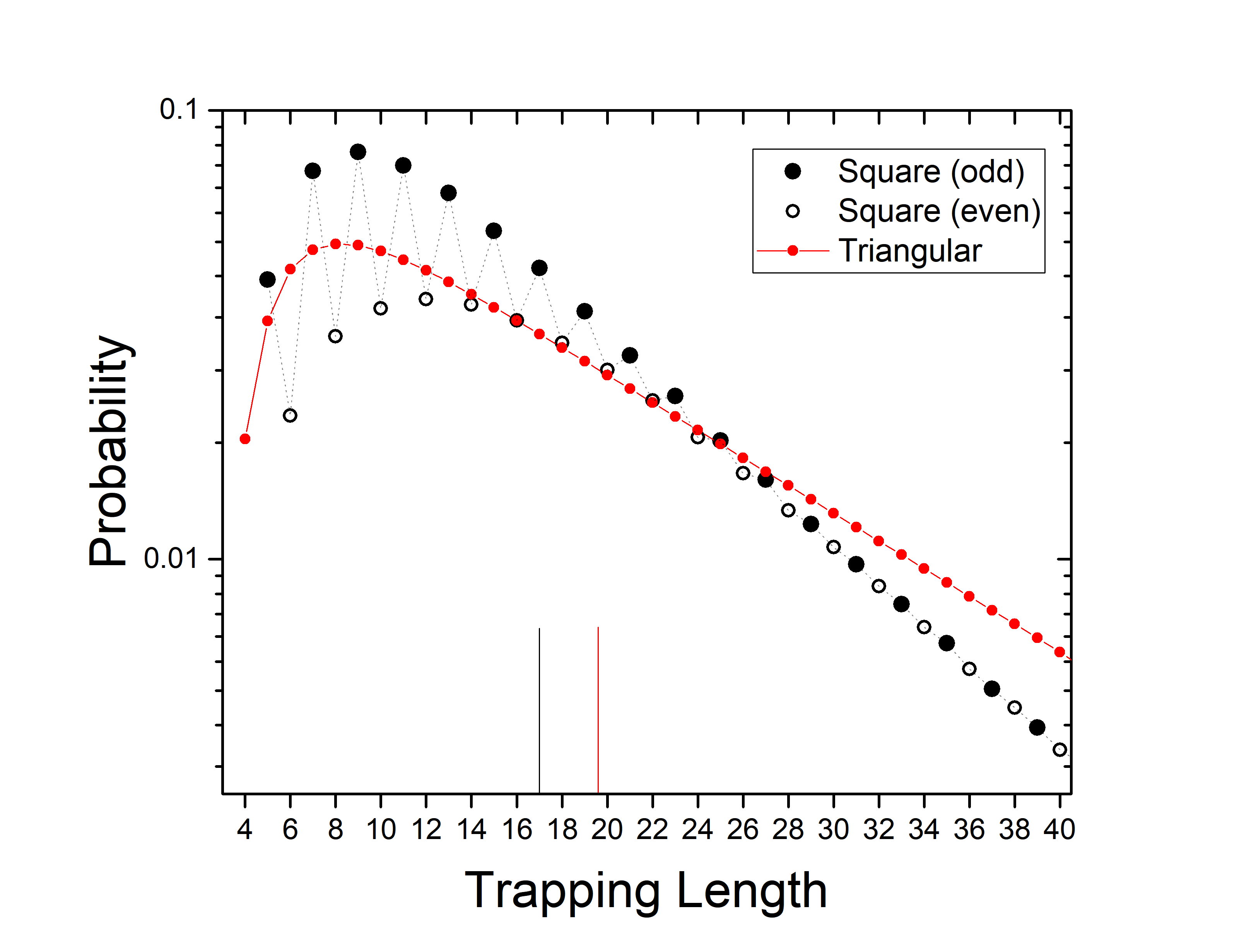}
    \caption{Trapping probability distribution of ladder graphs on the square and triangular lattices, according to Theorems~\ref{theorem:trapping-length-square-lattice} and~\ref{theorem:trapping-length-triangle-lattice}. Short walks on the square lattice are much more likely to become trapped after an odd number of steps. The trapping probability decays exponentially as approximately $0.93^N$ for the triangular case and $0.90^N$ for the square case. The vertical lines indicate the mean trapping lengths.}
    \label{fig:trap2}
\end{figure}


The decay constant of the exponential tails of the probability distributions can be exactly solved as the reciprocals of the poles of the generating functions that are closest to the origin \cite{wilf2005generatingfunctionology}. For the square lattice, the relevant pole is the reciprocal of the largest root of the cubic polynomial in the denominator of the generating function. In the long-walk limit, the square lattice probability distribution decays as $\left(\cos(\pi/7)\right)^N\approx 0.901^N$. For the triangular lattice, the relevant pole is the reciprocal of the largest real root of the quartic polynomial in the denominator of the generating function.  In the long-walk limit, the triangular lattice probability distribution decays as approximately $0.93^N$. The weaker exponential decay on the triangular lattice is consistent with a longer mean trapping length, as well as the asymptotic behavior on infinite lattices, seen in the inset of Figure \ref{fig:trapinf}.

The generating functions for the square and triangular lattices are equivalent to inhomogeneous recursion relations. The square lattice has a third-order recursion relation with alternating even and odd inhomogeneous terms decaying as a negative power of two:
\begin{equation}
P_{N}=\frac{P_{N-1}}{2}+\frac{P_{N-2}}{2}-\frac{P_{N-3}}{8}+
  \begin{cases}
   -\left(N-6\right)\cdot 2^{-(\frac{N+4}{2})}& \text{Even N} \\
    \ \ \left(N-3\right)\cdot 2^{-(\frac{N+3}{2})}  & \text{Odd N}   
  \end{cases}.
\end{equation}
The triangular lattice has a fourth-order recursion relation with strictly positive inhomogeneous terms decaying as a negative power of three:

\begin{equation}
  P_{N}=\frac{5}{6}P_{N-1}-\frac{1}{18}P_{N-2}+\frac{1}{6}P_{N-3}-\frac{1}{36}P_{N-4}+\frac{1}{16}\cdot
  \begin{cases}
    \left(\frac{5}{2}N+7\right)\cdot 3^{-(\frac{N+2}{2})}& \text{Even N} \\
    \left(\frac{3}{2}N+\frac{7}{2}\right)\cdot 3^{-(\frac{N+1}{2})}  & \text{Odd N}
  \end{cases}.
\end{equation}

The initial values in these recursion relations are 1/24, 1/48, and 7/96 for N=5, 6, 7 in the square case and 1/54, 11/324, 43/972, and 95/1944 for N=4, 5, 6, 7 for the triangular case.  The characteristic polynomials of each homogeneous component of the recursion relations have zeros that describe the asymptotic behavior discussed previously. The inhomogeneous components, representing the contribution to the probability of ladder graphs that cannot be constructed by appending steps to smaller graphs, highlight the striking parity effects seen only in the square lattice. Both inhomogeneous components vanish with large N, reducing to the asymptotic homogeneous trend. In the square lattice, the inhomogeneous terms alternate between positive and negative, putting the probability on either side of a trendline. The triangular inhomogeneous terms are strictly positive, and decay more strongly as a power of one-third rather than one-half, resulting in a smooth probability distribution.

The weighted walks on a square lattice have a mean trapping length (Corollary 2.61) that is non-monotonic with respect to the trapping length and is plot in Fig. \ref{fig:bias}. There is a global minimum at $C\approx1.63$ which in the thermodynamic formalism \cite{klotz} corresponds to approximately $kT/2$,\footnote{kT being the thermal energy scale, the product of temperature $T$ and Boltzmann constant $k$.} slightly larger than for the infinite square lattice which has a global minimum at $\approx kT/3$ or $C$=1.4. The asymptotic behavior of Corollary 2.61 is that the mean trapping length grows as $\langle N\rangle=2C$, which is equivalent to the exponential behavior seen in the infinite lattice. Corollary 2.61 also admits values of $C$ below 1, corresponding to walks that avoid self-adjacency. The mean trapping length increases asymptotically as $C$ approaches zero, and may be interpreted as the behavior as a polymer with strong electrostatic (rather than excluded volume) repulsion. Each probability distribution for a weighted lattice walk has an exponential decay constant that is the reciprocal of a quartic root (for $C\neq$ 1). The decay constant has a global minimum at a value of $C\approx1.75$, slightly greater than the position of the global minimum of the mean trapping length , $C\approx1.63$.

\begin{figure} 
    \centering
    \includegraphics[width=0.8\textwidth]{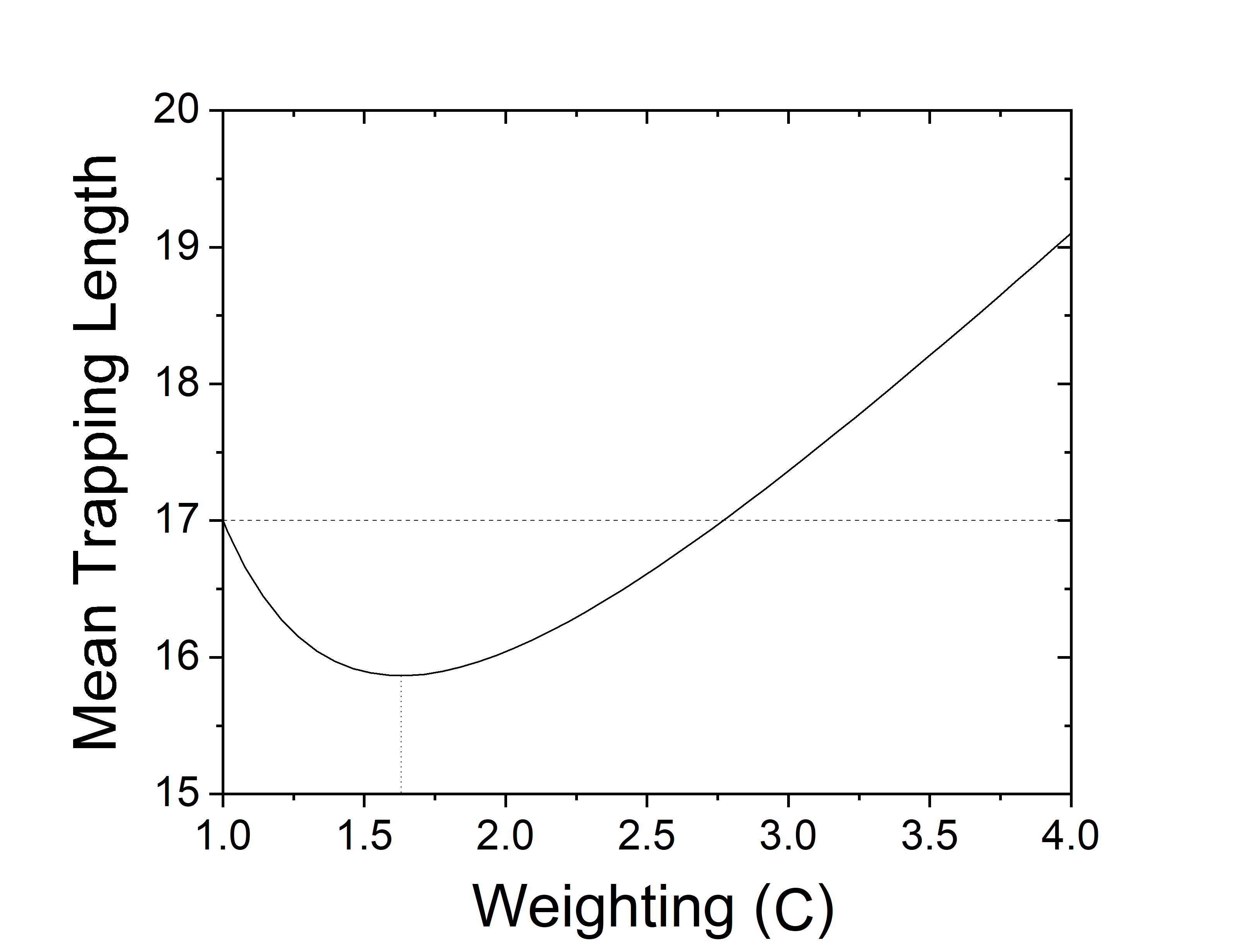}
    \caption{Mean trapping length as a function of the weighting factor $C$ for biased walks according to Corollary 2.61. There is a global minimum at $C=1.63$ and the growth is asymptotically linear in $C$.}
    \label{fig:bias}
\end{figure}

\section{Conclusion} We have developed an exactly solvable model of the growing self-avoiding walk. It allows us to exactly derive the mean trapping length of walks on square and triangular lattices that are two sites high. Additionally, we have shown that the exact probability distributions mirror the behavior seen in infinite square and triangular lattices, as well as empirical behavior observed for biased GSAWs. Restricting the lattice to two sites in height prevents combinatorial explosion while admitting non-trivial behavior. Strips of height greater than two are significantly more complex and will be discussed in Part II~\cite{saws-part2}. We hope that this work inspires further investigation into exact solutions for stochastic systems in statistical mechanics.


\section{Acknowledgements}
We are grateful to Jay Pantone for insightful discussions during the writing of this manuscript. We also acknowledge helpful suggestions from Tony Guttmann and Nathan Clisby. ARK is supported by the National Science Foundation, grant number 2105113.

\section{Appendix}

\subsection{Infinite Lattices: numerical results}

To compare the ladder walks with those of unrestricted lattices, we compute the trapping probability distribution on square, triangular, and honeycomb lattices. A walk was initiated in the center of a 512 by 512 matrix, with steps to unoccupied adjacent sites chosen by a random number generator in MATLAB. Previously, the trapping distribution of 60,000 walks was reported by Renner \cite{renner1996self}. Here, we present histograms of the trapping distributions after 1,000,000 simulated walks. Approximately ten walks out of each million reached the wall of simulation box, which we do not consider to have a significant effect.

The even-odd asymmetry of the square lattice has been established \cite{hemmer1984average,klotz}. The triangular lattice lacks this asymmetry and has a trapping probability distribution that is much smoother. Trapping in the honeycomb lattice is most likely at a multiple of four plus one. The square and honeycomb probability distributions have very similar exponential behavior, while the triangular case decays with a weaker constant.

\begin{figure}
    \centering
    \includegraphics[width=0.8\textwidth]{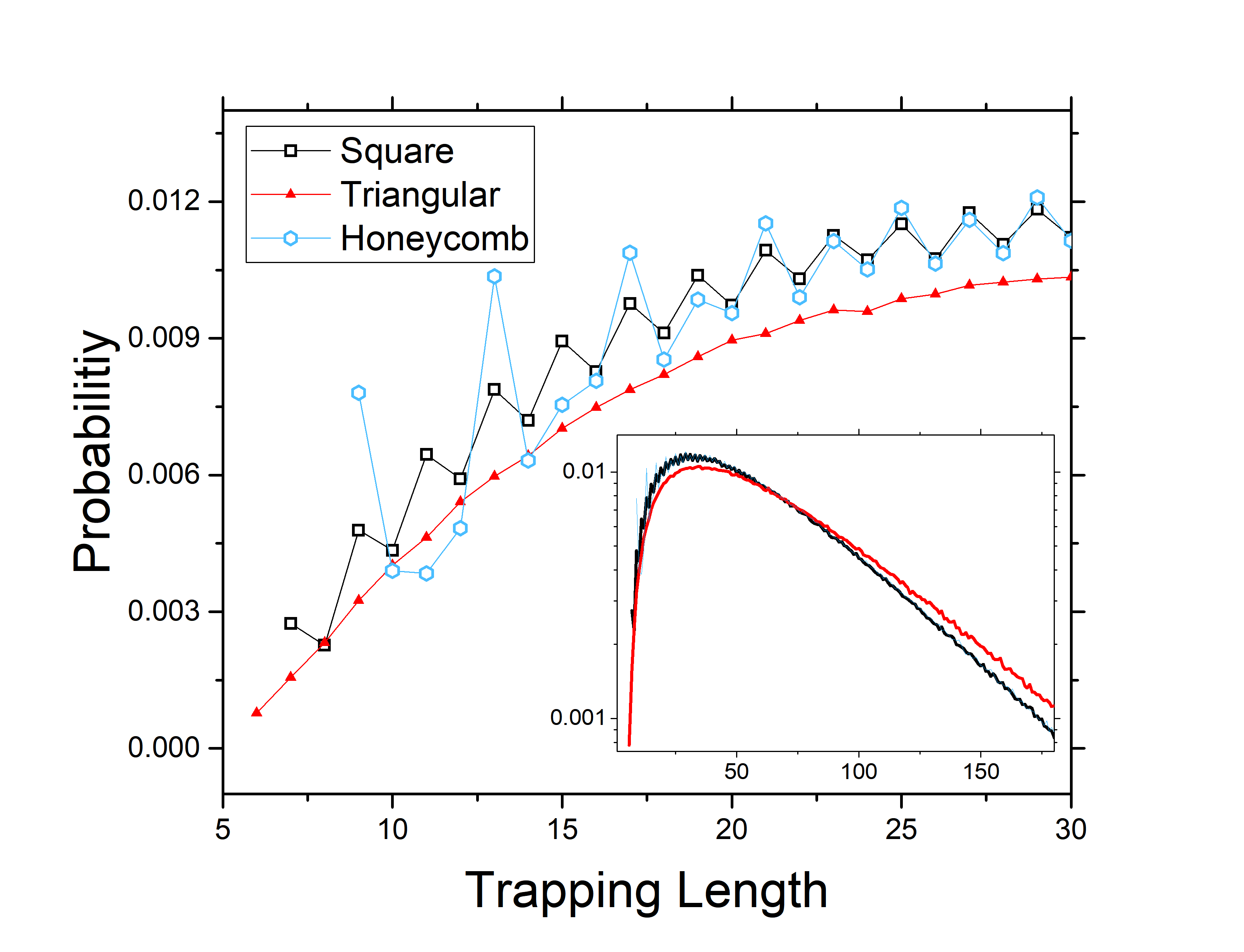}
    \caption{Trapping probability distribution for GSAWs on the square, triangular, and honeycomb lattices. Inset shows the same data on semi-logarithmic axes. For walks trapped after a short distance, square lattice walks are more likely to become trapped after an odd number of steps, the triangular lattice imposes a smooth distribution, and walks on the honeycomb lattice is most likely to be trapped at a multiple of four plus one. For walks trapped after a long distance, the square and honeycomb lattices have the same asymptotic exponenitial behavior while the triangular lattice has a weaker exponential decay.}
    \label{fig:trapinf}
\end{figure}


\bibliographystyle{plainnat}
\bibliography{walkrefs}

\end{document}